\documentclass[12pt, reqno]{amsart}
\usepackage{amsmath, amsthm, amscd, amsfonts, amssymb, graphicx, color, mathrsfs}
\usepackage[bookmarksnumbered, colorlinks, plainpages]{hyperref}
\usepackage[all]{xy} 
\usepackage{slashed}

\newtheorem{theorem}{Theorem}[section]
\newtheorem{lemma}[theorem]{Lemma}

\theoremstyle{definition}

\theoremstyle{remark}
\newtheorem{remark}[theorem]{Remark}
\numberwithin{equation}{section}

\begin{document}
\setcounter{page}{1}

\title[Oscillating singular integrals on  Lie groups of pol. growth]{Boundedness of oscillating singular integrals on Lie groups of  polynomial growth}

\author[D. Cardona]{Duv\'an Cardona}
\address{
  Duv\'an Cardona:
  \endgraf
  Department of Mathematics: Analysis, Logic and Discrete Mathematics
  \endgraf
  Ghent University, Belgium
  \endgraf
  {\it E-mail address} {\rm duvanc306@gmail.com, duvan.cardonasanchez@ugent.be}
  }
  
\author[M. Ruzhansky]{Michael Ruzhansky}
\address{
  Michael Ruzhansky:
  \endgraf
  Department of Mathematics: Analysis, Logic and Discrete Mathematics
  \endgraf
  Ghent University, Belgium
  \endgraf
 and
  \endgraf
  School of Mathematical Sciences
  \endgraf
  Queen Mary University of London
  \endgraf
  United Kingdom
  \endgraf
  {\it E-mail address} {\rm michael.ruzhansky@ugent.be, m.ruzhansky@qmul.ac.uk}
  }

\thanks{The authors are supported  by the FWO  Odysseus  1  grant  G.0H94.18N:  Analysis  and  Partial Differential Equations and by the Methusalem programme of the Ghent University Special Research Fund (BOF)
(Grant number 01M01021). Michael Ruzhansky is also supported  by EPSRC grant 
EP/R003025/2.
}


\begin{abstract} We investigate the boundedness of oscillating singular integrals on Lie groups of polynomial growth in order to extend the classical oscillating conditions due to Fefferman and Stein for the boundedness of oscillating convolution operators. Kernel criteria are presented in terms of a  fixed sub-Riemannian structure on the group induced by a sub-Laplacian associated to a  H\"ormander system of vector fields. In the case where the group is graded, kernel criteria are presented in terms of the Fourier analysis associated to an arbitrary Rockland operator. 
\end{abstract} 

\maketitle

\tableofcontents
\allowdisplaybreaks

\section{Introduction}

\subsection{Outline and historical remarks}

In this paper we study the boundedness of oscillating singular integrals on a connected Lie group $G$  of polynomial growth. We address our investigation by considering  two settings. 
\begin{itemize}
    \item First, we  consider on $G$ a positive sub-Laplacian $\mathcal{L}=-\sum_{i=1}^kX_i^2$ associated with a systems of vector fields $X=\{X_i\}_{i=1}^k$ satisfying the H\"ormander condition. Then, our kernel criteria for the boundedness of operators are presented in terms of the sub-Laplacian $\mathcal{L}$ and of the kernel of the operator.
    \item Later, we consider the case where    $G$ is a  graded Lie group, but  generalising the sub-Laplacian context, the kernel criteria are presented in terms of the Fourier analysis associated with an arbitrary Rockland operator $\mathcal{R}$.
\end{itemize}
In both cases, implicitly, our arguments make use of the hypoellipticity of the sub-Laplacian $\mathcal{L}$ and of the Rockland operator $\mathcal{R}$  in view of  the H\"ormander theorem of the  sums of squares \cite{Hormander1967}, and of the Helffer and Nourrigat solution of the Rockland conjecture \cite{HelfferNourrigat}, respectively.  

As a consequence of our investigation, further applications can be obtained for many classes of Lie groups of interest in analysis and geometry e.g. the Heisenberg group, the Cartan group, the Engel group,  other stratified groups, arbitrary compact Lie groups such as the torus $\mathbb{T}^n,$  $\textnormal{SU}(2)\cong \mathbb{S}^3,$   $\textnormal{SO}(3),$ graded Lie groups, and general Lie groups of polynomial growth.

In analysis, the theory of singular integrals has a long tradition (see e.g.  Mihlin \cite{Mihlin} and references therein).  This class of operators was introduced in the study of elliptic problems, for instance, its importance in the proof of  the celebrated Calder\'on theorem on the uniqueness of the Cauchy problem has given testimony of its relevant applications to the theory of PDE, see \cite{Calderon1958} for details.  On the other hand, the program of Calder\'on and  Zygmund (which started with their memoir \cite{CalderonZygmund1952}) was dedicated to generalising several of the fundamental results of the one-dimensional harmonic analysis to higher dimensions (as the weak (1,1)-boundedness theorem for the Hilbert tranform due to Kolmogorov, and its $L^p$-boundedness theorem due to M. Riesz).   Being the  Hilbert transform, the fundamental one-dimensional example of a singular integral,  Calder\'on and Zygmund in \cite{CalderonZygmund1952} investigated convolution operators of the form
\begin{equation}\label{CZ:1952}
    Tf(x)=\textnormal{p.v.}\smallint_{\mathbb{R}^n}K(x-y)f(y)dy,
\end{equation}where the kernel $K,$ having some regularity properties, is homogeneous of degree $-n,$ and satisfies the cancellation property  $\smallint_{|x|=1}K(x)d\sigma(x)=0.$ Here, $d\sigma$ is the surface density on the unit sphere.

A well known condition for the boundedness properties of singular integrals is the one due to H\"ormander  \cite{Hormander1960}, where the distributional kernels $K$ are taken satisfying  a ``smoothness condition" of the form
\begin{eqnarray}\label{HC}
& [K]_{H_\infty}:= \sup_{0<R<1}\sup_{|y|<R}\smallint_{|x|\geq 2R}|K(x-y)-K(x)|dx <\infty.  
\end{eqnarray}
Indeed, the H\"ormander condition in \eqref{HC} together with the  Fourier transform condition $\widehat{K}\in L^\infty$ give sufficient conditions for the weak (1,1) boundedness of $T$ (as in \ref{CZ:1952}) and also for its $L^p$-boundendness in the range $1<p<\infty$.

Other conditions generalising the one by  H\"ormander, arise in the analysis of  oscillating multipliers (see Hardy \cite{Hardy1913}, Hirschman \cite{Hirschman1956} and Wainger \cite{Wainger1965}). By a suggestion of Stein,  Fefferman in  \cite{Fefferman1970} has considered  distributions satisfying the ``oscillating smoothness" condition
\begin{equation}\label{FeffCond}
   [K]_{H_{\infty},\theta}:=  \sup_{0<R<1}\sup_{|y|<R} \smallint\limits_{|x|\geq 2R^{1-\theta}}|K(x-y)-K(x)|dx   <\infty.
\end{equation} Indeed, roughly speaking, if $K$ satisfies \eqref{FeffCond} and its Fourier transform has decay
\begin{equation}\label{decay}
    |\widehat{K}(\xi)|=O((1+|\xi|)^{-\frac{n\theta}{2}}),\quad 0\leq \theta<1,
\end{equation}
Fefferman's theorem says that $T$ admits a bounded extension of weak (1,1) type. Obviously, with $\theta=0,$ Fefferman's condition agrees with the one in \eqref{HC} by H\"ormander. Observe that Fefferman's approach illustrates the delicate relationship between the condition on $K$ and its Fourier transform $\widehat{K}.$

Further extensions of the Calder\'on-Zygmund theory to non-commutative structures were developed e.g. by Coifman and Weiss in \cite{CoifmanWeiss} and by Coifman and De Guzm\'an in  \cite{CoifmandeGuzman} for spaces of homogeneous type just to mention a few.

\subsection{The main result}

 In this work we prove that the Fefferman conditions for oscillating kernels beget $L^p$-bounded operators not only in the Euclidean setting, but also in  general non-commutative structures, namely, connected Lie groups of polynomial growth. In particular, for  graded Lie groups (which are also Lie groups of polynomial growth but with an extra structure of dilations on its graded Lie algebra) we will use the existence of a general class of hypoelliptic differential operators that in this context are called Rockland operators, see e.g. Helfer and Nourrigat \cite{HelfferNourrigat} or \cite[Chapter IV]{FischerRuzhanskyBook} for details.  

In order to present our main result (see Theorem \ref{main:th:2} below), we introduce some preliminary notions. If the Lie group is of type I (e.g. any compact Lie group, any graded Lie group, or in general homogeneous Lie groups) we denote by $\widehat{G}$  the unitary dual of $G,$ which consists of all equivalence classes of unitary and irreducible representations on $G.$  In this setting, $[\xi]$ denotes an equivalence class in the unitary dual $\widehat{G}$. The representation spaces $H_\xi$ associated to $\xi:G\rightarrow \mathcal{B}(H_\xi)$ are separable Hilbert spaces (with finite dimension in the case where $G$ is compact). Here, $\mathcal{B}(H_\xi)$ denotes the space of the bounded linear operators on $H_\xi.$  The representation $\xi$ becomes operator-valued, and then, the Fourier transform of a distribution $K$ on $G,$ is given by
\begin{equation}
    \widehat{K}(\xi)=\smallint_{G}K(x)\xi(x)^*dx:H_\xi\rightarrow H_\xi.
\end{equation}
As usually, $dx$ denotes the left-invariant Haar measure associated to $G.$

In the case of a graded Lie group, the unitary dual $\widehat{G}$ becomes to be a purely continuous set, and by following the standard notation for the Fourier analysis on graded Lie groups (taken from the formalism of $C^*$-algebras)  the equivalence classes of unitary representations will be denoted by $[\pi]\in \widehat{G}$ (instead of $[\xi]$). Also, in the  setting of graded Lie groups the class $[\pi]$ is typically identified with a representation $\pi$ of the class. The representation space $H_\pi$ is $\pi$-a.e. a separable Hilbert space of infinite dimension. Nevertheless, for the variety of examples  presented in Section  \ref{Examples}, e.g. in the case of compact Lie groups, the notation $[\xi]\in \widehat{G}$ will be reserved.

On an arbitrary Lie group of polynomial growth, we will show (as in Fefferman's criterion above) that the Fourier transform condition \eqref{decay} on the kernel $K,$ can be replaced by the $L^2$-boundendess of the operator in \eqref{L:2:hypo} below,  and that the following  conditions  
\begin{equation}\label{c1:intro}
        [K]_{H_{\infty,\theta}}:=\sup_{0<R<1}\sup_{|y|<R} \smallint\limits_{|x|\geq 2R^{1-\theta}}|K(xy^{-1})-K(x)|dx   <\infty,
    \end{equation} and     
    \begin{equation}\label{c2:intro}
        [K]_{H_{\infty,\theta}}':=\sup_{0<R<1}\sup_{|y|<R}  \smallint\limits_{|x|\geq 2R^{1-\theta}}|K(y^{-1}x)-K(x)|dx  <\infty,
    \end{equation} can be used to study the $L^p$-boundedness of the corresponding singular integral  operator $Tf:=K\ast f.$ We have denoted by $$B(R)=\{x\in G:|x|<R\}$$  the ball of radius $R>0$ centred at the identity and by $|B(R)|:=\textnormal{Vol}(B(R))$  its Haar measure. 
    
    The distance $|\cdot|$ used in \eqref{c1:intro} or in \eqref{c2:intro} depends of the geometry of the group $G$ under consideration. Indeed, for Lie groups of polynomial growth endowed with a H\"ormander sub-Laplacian the natural  distance is the one defined by using horizontal curves. On the another hand, if the group is  graded we chose a quasi-distance on the group compatible with its structure of dilations.

We record that if $|\cdot|$ is the Carnot-Carath\'eodory distance on a Lie group of polynomial growth $G,$ (see Coulhon,  Saloff-Coste, and Varopoulos \cite{Saloff-CosteBook}) then there exist $d=d_0$ and $d'=d_\infty$ such that 
    
\begin{equation}
    \forall R\in [0,1],\, |B(x,R)|\sim R^{d_0},\,\forall R\geq 1,\, |B(x,R)|\sim R^{d_\infty},
\end{equation} where $B(x,R)$ is the ball of radius $R>0$ associated to the distance $|\cdot|.$   Then $d_0$ and $d_\infty$ are called  the local dimension and the global dimension of the pair $(G,X),$ respectively. 

It can be proved that the global dimension depends only of the structure of the group and it is independent of the system $X,$ see Coulhon,  Saloff-Coste, and Varopoulos \cite{Saloff-CosteBook} for details. However, the local dimension $d_0$ depends of both, of the group $G$ and of the vector fields  system $X.$ One of the main aspects of our work, is that our kernel criteria are presented in terms of the local dimension $d_0.$

When asking how to generalise the condition \eqref{decay} to the case of  Lie groups of polynomial growth, one can give a possible solution using the Fourier transform of the group $G$ if it is a Lie group of type I, see Subsection \ref{LGPGTI} for details. Nevertheless, the hypothesis that the  Lie group $G$ is of type $I$ can be removed if we consider the group to be endowed with a sub-Riemannian structure induced by a family of  vector-fields $\{X_j\}_{j=1}^k,$ satisfying the H\"ormander condition,  that is, the family of vector-fields $\{X_j\}_{j=1}^k$ and their iterated commutators generate the Lie algebra $\mathfrak{g}\sim T_{e}G$ of $G$. In this general setting, \eqref{decay} can be extended to general Lie groups of polynomial growth just by assuming that the operators
\begin{equation}\label{L:2:hypo}
      (1+{\mathcal{L}})^{\frac{d_{0}\theta}{4}}T,\,  T(1+{\mathcal{L}})^{\frac{d_{0}\theta}{4}}:L^2(G)\rightarrow L^2(G),
    \end{equation} admit   bounded extensions.

Now, when $G$ is a graded Lie group, the Euclidean  condition \eqref{decay} can be extended by using the infinitesimal representation of any Rockland operator. If $G$ is a stratified Lie group the prototype of Rockland operators are sub-Laplacians $\mathcal{L}=-\sum_{i=1}^kX_i^2.$ However, Rockland operators may have an arbitrary order. For instance, on the Heisenberg group $\mathbb{H}_n\sim \mathbb{R}^{2n+1}$ instead of the positive sub-Laplacian  $\mathcal{L}=-\sum_{i=1}^n(X_i^2+Y_i^2)$ we can choose e.g. the positive Rockland operator $\mathcal{R}=\sum_{i=1}^n(X_i^4+Y_i^4)$ of fourth order. For   examples of Rockland operators with an arbitrary order we refer the reader to Corollary 4.1.8 of \cite{FischerRuzhanskyBook}.

If we fix a  Rockland operator $\mathcal{R}$ and if $\tau$ is the right-convolution kernel associated with $\mathcal{R},$ $\mathcal{R}f=f\ast \tau,$ then
\begin{equation}\label{pi:R}
  \pi(\mathcal{R}):=\widehat{\tau}(\pi)\equiv d\pi(\mathcal{R})
\end{equation}
 denotes its Fourier transform, and  the condition in  \eqref{decay} takes in the  setting of a graded Lie group $G\cong \mathfrak{g}_0\times \mathfrak{g}_1\times \cdots \mathfrak{g}_s $ the following form 
$$ \sup_{\pi\in \widehat{G}}\Vert \widehat{K}(\pi) (1+\pi(\mathcal{R}))^{\frac{Q\theta}{2\nu}}\Vert_{\textnormal{op}},\,  \sup_{\pi\in \widehat{G}}\Vert (1+\pi(\mathcal{R}))^{\frac{Q\theta}{2\nu}}\widehat{K}(\pi) \Vert_{\textnormal{op}}<\infty,  $$
where $\nu=\nu_{\mathcal{R}}>0$ denotes the homogeneity degree of $\mathcal{R}.$ As before,  $\widehat{K}(\pi)$ denotes the Fourier transform of a distribution $K$ on $G.$ The main result of this work is the following theorem.

\begin{theorem}\label{main:th:2} Let $G$ be a Lie group of polynomial growth. Let $T $ be a convolution operator with a  distribution $K\in \mathscr{D}'(G),$ and let $0\leq \theta<1$. The following statements hold true.
\begin{itemize}
     \item Let $X=\{X_i\}_{i=1}^k$ be a H\"ormander system of vector fields and let $\mathcal{L}=-\sum_{j=1}^{k}X_{j}^2$ be the associated  positive sub-Laplacian. Let $|\cdot|$ be the Carnot-Carath\'eodory distance on $G$ associated to $X,$ and let   $d=d_0$ be the local  dimension of $(G,X).$ Assume that 
    \begin{equation}\label{Fouriergrowth:2}
      (1+{\mathcal{L}})^{\frac{d_{0}\theta}{4}}T:C^\infty_0(G)\subset L^2(G)\rightarrow L^2(G),
    \end{equation} admits a bounded extension. If additionally, $K$ satisfies  the kernel condition 
    \begin{equation}\label{GS:CZ:cond:2'}
        [K]_{H_{\infty,\theta}}:=\sup_{0<R<1}\sup_{|y|<R}  \smallint\limits_{|x|\geq 2R^{1-\theta}}|K(xy^{-1})-K(x)|dx   <\infty,
    \end{equation}  then $T:L^\infty(G)\rightarrow BMO(G)$ admits a bounded extension. On the other hand, if 
    \begin{equation}\label{Fouriergrowth:2New2}
      T(1+{\mathcal{L}})^{\frac{d_{0}\theta}{4}}:C^\infty_0(G)\subset L^2(G)\rightarrow L^2(G),
    \end{equation} is bounded and $K$ satisfies  the kernel condition   
    \begin{equation}\label{GS:CZ:cond:22'}
        [K]_{H_{\infty,\theta}}':=\sup_{0<R<1}\sup_{|y|<R}  \smallint\limits_{|x|\geq 2R^{1-\theta}}|K(y^{-1}x)-K(x)|dx  <\infty,
    \end{equation} then $T:H^1(G)\rightarrow L^1(G)$ extends to a bounded operator.
    \item Consider $G$ to be a  graded Lie group (and hence non-compact), let $|\cdot|$ be a homogeneous quasi-norm on $G$ and let $Q$ be  its homogeneous dimension. Let $\mathcal{R}$ be a Rockland operator of homogeneous degree $\nu>0.$ Assume that  $K$ satisfies the estimate
    \begin{equation}\label{Fourier:growth:1}
        \sup_{\pi\in \widehat{G}}\Vert (1+\pi(\mathcal{R}))^{\frac{Q\theta}{2\nu}}\widehat{K}(\pi) \Vert_{\textnormal{op}}<\infty.
    \end{equation} If additionally, $K$ satisfies the kernel condition 
     \begin{equation}\label{GS:CZ:cond:2}
        [K]_{H_{\infty,\theta}}:=\sup_{0<R<1}\sup_{|y|<R}  \smallint\limits_{|x|\geq 2R^{1-\theta}}|K(xy^{-1})-K(x)|dx   <\infty,
    \end{equation}  then $T:L^\infty(G)\rightarrow BMO(G)$ admits a bounded extension. On the other hand, if one has the estimate
     \begin{equation}\label{Fourier:growth:1New}
        \sup_{\pi\in \widehat{G}}\Vert \widehat{K}(\pi) (1+\pi(\mathcal{R}))^{\frac{Q\theta}{2\nu}} \Vert_{\textnormal{op}}<\infty,
    \end{equation} 
    and the kernel condition   
    \begin{equation}\label{GS:CZ:cond:22}
        [K]_{H_{\infty,\theta}}':=\sup_{0<R<1}\sup_{|y|<R}  \smallint\limits_{|x|\geq 2R^{1-\theta}}|K(y^{-1}x)-K(x)|dx  <\infty,
    \end{equation} then $T:H^1(G)\rightarrow L^1(G)$ extends to a bounded operator.
\end{itemize}
In any case above,  $T$ admits a bounded extension on $L^p(G)$ for all $2\leq p<\infty,$ provided that $K$ satisfies \eqref{Fouriergrowth:2} and \eqref{GS:CZ:cond:2'}, or \eqref{Fourier:growth:1} and   \eqref{GS:CZ:cond:2}, respectively. Moreover,  $T$ admits a bounded extension on $L^p(G)$ for all $1<p\leq 2,$ provided that    $K$ satisfies  \eqref{Fouriergrowth:2New2} and \eqref{GS:CZ:cond:22'}, or \eqref{Fourier:growth:1New} and \eqref{GS:CZ:cond:22}, respectively.
\end{theorem}
Now, we briefly present some remarks related with  Theorem \ref{main:th:2}.
\begin{remark}We note that the statement in Theorem  \ref{main:th:2} for $\theta=0$ follows from the general theory of singular integrals developed by Coifman and Weiss \cite{CoifmanWeiss} for spaces of homogeneous type. So, the contribution in Theorem  \ref{main:th:2} corresponds to the oscillating setting, namely, when $0<\theta<1.$
\end{remark}

\begin{remark}
In the context of $G=\mathbb{R}^n,$ and of the positive Laplacian $\mathcal{R}=\mathcal{L}=-\Delta_x,$ the $H^1(\mathbb{R}^n)-L^1(\mathbb{R}^n)$-boundedness result was proved in the non-oscillating case $\theta=0$  by Suzuki in \cite{Suzuki2021}. As mentioned above, in the Euclidean case that the conditions $[K]_{\infty,\theta}<\infty$ and \eqref{decay} are enough to deduce the $H^1(\mathbb{R}^n)-L^1(\mathbb{R}^n)$ boundedness of the corresponding convolution operator $Tf=f\ast K$ were proved by Fefferman and Stein in \cite{FeffermanStein1972} when $0\leq \theta <1.$ 

\end{remark}
\begin{remark} 
In the non-commutative case, the $L^1$-average conditions in Theorem \ref{main:th:2} are already new, even considering Lie groups of polynomial growth or the graded case. Observe that in the setting of Lie groups of polynomial growth $G$ both of the situations $d_0\leq d_\infty $ and $d_\infty\leq d_0$ are equally probable. For instance, if  $G$ is a simply connected nilpotent Lie group, then $d_0\leq d_\infty,$ and when $G$ is compact $d_\infty=0.$ 

An interesting fact is that starting with a Lie group $G$ where $d_0\leq d_\infty$ we can always consider the product space $G'=\mathbb{T}^{d_\infty-d_0+1}\times G,$ which has local dimension $d_0'=(d_\infty-d_0+1)+d_0=d_\infty+1,$ global dimension $d_\infty'=0+d_\infty=d_\infty$ and of course $d_\infty'<d_0'.$ Let us finalise this remark by observing that in the case of a  stratified nilpotent  Lie group $G,$  $d_0=d_\infty,$ see Coulhon,  Saloff-Coste, and Varopoulos \cite[Chapter VI]{Saloff-CosteBook}. {{ Also, if $G$ is a graded Lie group, then $d_0=d_\infty=Q,$ the homogeneous dimension of $G$. }}
\end{remark}
\begin{remark}
In our analysis the appearance of the local dimension $d_0$ is justified by the estimates of  the heat kernel $e^{-t\mathcal{L}}$ of the sub-Laplacian for small time $t\rightarrow 0^+$, see Coulhon,  Saloff-Coste, and Varopoulos \cite[Chapter VIII]{Saloff-CosteBook}. 
\end{remark}
\begin{remark}
For $\theta=0$ the conditions in \eqref{Fouriergrowth:2} and \eqref{Fourier:growth:1} are reduced to  the $L^2$-boundedness of the operator $T,$ in view of the group Plancherel theorem. For singular integrals on $\mathbb{R}^n$ the $L^2$-boundedness of singular integrals (of possibly non-convolution type) was characterised by David and Journ\'e in their celebrated work \cite{DavidJourne}.
\end{remark}
\begin{remark}
Fourier multipliers satisfying Marcinkiewicz, and H\"ormander-Mihlin criteria  in the context of Lie groups and several spaces of homogeneous type have Calder\'on-Zygmund type kernels. In the context of spectral multipliers of  self-adjoint operators, e.g., sub-Laplacians,  or of other operators with heat kernels  satisfying  Gaussian estimates, with general contexts that go beyond of the objective of this paper we refer the reader e.g. to   \cite{alexo,Anker,Sikora,CowlingSikora,CW3},  the book of Stein \cite{Stein} and to the extensive list of references therein.
\end{remark}

This paper is organised as follows. The paper ends with Section \ref{Examples} where we present some examples illustrating our main Theorem \ref{main:th:2}, however the notations and the notions (of representation theory and of Lie theory) that we use  are contained in Section \ref{Preliminaries}.  
In Section \ref{ProofMainTh} we present the proof of Theorem \ref{main:th:2}.

\section{Preliminaries}\label{Preliminaries}

\subsection{Sub-Laplacians  on  Lie groups of polynomial growth}
Let $G$ be a Lie group of  polynomial growth  with Lie algebra $\mathfrak{g}.$ Under the identification $\mathfrak{g}\simeq T_{e_G}G,  $ where $e_{G}$ is the identity element of $G,$ let us consider  a system of $C^\infty$-vector fields $X=\{X_1,\cdots,X_k \}\in \mathfrak{g}$. For all $I=(i_1,\cdots,i_\omega)\in \{1,2,\cdots,k\}^{\omega},$ of length $\omega\geqslant   2,$ denote $$X_{I}:=[X_{i_1},[X_{i_2},\cdots [X_{i_{\omega-1}},X_{i_\omega}]\cdots]],$$ and for $\omega=1,$ $I=(i),$ $X_{I}:=X_{i}.$ Let $V_{\omega}$ be the subspace generated by the set $\{X_{I}:|I|\leqslant \omega\}.$ That $X$ satisfies the H\"ormander condition,  means that there exists $\kappa'\in \mathbb{N}$ such that $V_{\kappa'}=\mathfrak{g}.$ Certainly, we consider the smallest $\kappa'$ with this property and we denote it by $\kappa$ which will be  later called the step of the system $X.$ We also say that $X$ satisfies the H\"ormander condition of order $\kappa.$ Note that  the sum of squares
\begin{equation*}
    \mathcal{L}=-\sum_{i=1}^kX_i^2,
\end{equation*} is a subelliptic operator which by following the usual nomenclature is called the subelliptic Laplacian associated with the family $X.$ For short we refer to $\mathcal{L}$ as the sub-Laplacian. In view of the H\"ormander theorem on sums of the  squares of vector fields (see H\"ormander \cite{Hormander1967}) it is a hypoelliptic operator (i.e. if $ \mathcal{L}u\in C^\infty(G)$ with $u\in  \mathscr{D}'(G)$ then $u\in C^\infty(G),$ and also locally at all points). 

A central notion in the analysis of the sub-Laplacian is that of the Hausdorff dimension, in this case, associated to $\mathcal{L}$. Indeed, for all $x\in G,$ denote by $H_{x}^\omega G$ the subspace of the tangent space $T_xG$ generated by the $X_i$'s and all the Lie brackets  $$ [X_{j_1},X_{j_2}],[X_{j_1},[X_{j_2},X_{j_3}]],\cdots, [X_{j_1},[X_{j_2}, [X_{j_3},\cdots, X_{j_\omega}] ] ],$$ with $\omega\leqslant \kappa.$ The H\"ormander condition can be stated as $H_{x}^\kappa G=T_xG,$ $x\in G.$ We have the filtration
\begin{equation*}
H_{x}^1G\subset H_{x}^2G \subset H_{x}^3G\subset \cdots \subset H_{x}^{\kappa-1}G\subset H_{x}^\kappa G= T_xG,\,\,x\in G.
\end{equation*} In our case,  the dimension of every $H_x^\omega G$ does not depend on $x$ and we write $\dim H^\omega G:=\dim H_{x}^\omega G,$ for any $x\in G.$ So, the Hausdorff dimension can be defined as (see e.g. \cite[p. 6]{HK16}),
\begin{equation}\label{Hausdorff-dimension}
    Q:=\dim(H^1G)+\sum_{i=1}^{\kappa-1} (i+1)(\dim H^{i+1}G-\dim H^{i}G ).
\end{equation}

\subsection{Lie groups of Type I  and group Plancherel formula} Let $G$ be a second countable  Lie group $G$ of type I and let $dx$ be  its left-invariant Haar measure.
Let us record the notion of the unitary dual $\widehat{G}$ of $G$.  Let us denote by  $\xi$  a strongly continuous, unitary and irreducible  representation of $G,$ this means that,
\begin{itemize}
    \item $\xi\in \textnormal{Hom}(G, \textnormal{U}(H_{\xi})),$ for some infinite-dimensional Hilbert space $H_\xi,$ i.e. $\xi(xy)=\xi(x)\xi(y)$ and for the  adjoint of $\xi(x),$ $\xi(x)^*=\xi(x^{-1}),$ for every $x,y\in G.$
    \item The map $(x,v)\mapsto \xi(x)v, $ from $G\times H_\xi$ into $H_\xi$ is continuous.
    \item For every $x\in G,$ and $W_\xi\subset H_\xi,$ if $\xi(x)W_{\xi}\subset W_{\xi},$ then $W_\xi=H_\xi$ or $W_\xi=\emptyset.$
\end{itemize} Let $\textnormal{Rep}(G)$ be the set of unitary, strongly continuous and irreducible representations of $G.$ The relation, {\small{
\begin{equation*}
    \xi_1\sim \xi_2\textnormal{ if and only if, there exists } A\in \textnormal{End}(H_{\xi_1},H_{\xi_2}),\textnormal{ such that }A\xi_{1}(x)A^{-1}=\xi_2(x), 
\end{equation*}}}for every $x\in G,$ is an equivalence relation and the unitary dual of $G,$ denoted by $\widehat{G}$ is defined via
$$
    \widehat{G}:={\textnormal{Rep}(G)}/{\sim}.
$$

If additionally, $G$ is a Lie group of polynomial growth, for every representation $\xi\in\widehat{G},$ $\xi:G\rightarrow \mathscr{B}({H}_{\xi}),$ we denote by ${H}_{\xi}^{\infty}$ the set of smooth vectors, that is, the space of elements $v\in {H}_{\xi}$ such that the function $x\mapsto \xi(x)v,$ $x\in {G},$ is smooth.

The Fourier transform $\widehat{f}$ of a distribution $f\in \mathscr{D}'(G)$ is defined via,
\begin{equation*}
    \widehat{f}(\xi)\equiv (\mathscr{F}f)(\xi):=\smallint\limits_{G}f(x)\xi(x)^*dx.
\end{equation*} If $f\in L^1(G),$ then for a.e. $\xi,$ $\widehat{f}(\xi)\in \mathscr{B}(H_\xi).$ Because $G$ is of type I, the unitary dual $\widehat{G}$ admits a Borel structure $(\widehat{G},d\nu(\xi))$ called in this setting the Mackey-Borel structure. Indeed, by denoting $\mathscr{B}^2(\mathcal{H}_\xi)$ the class of Hilbert-Schmidt operators on $\mathcal{H}_\xi,$ one has the direct Hilbert space integral
\begin{equation}
    \mathscr{B}^2(\widehat{G}):=\smallint\limits_{\widehat{G}}^{\oplus}\mathscr{B}^2(H_\xi)d\nu(\xi)\cong \smallint\limits_{\widehat{G}}^{\oplus}H_\xi\otimes \overline{H}_\xi d\nu(\xi). 
\end{equation}
In particular, the Fourier transform $\mathscr{F}:L^1(G)\rightarrow \mathscr{B}(\widehat{G}):= \smallint\limits_{\widehat{G}}^{\oplus}\mathscr{B}(H_\xi)d\nu(\xi)$ is injective and one can reconstruct a function using  the Fourier inversion formula
\begin{equation*}
    f(x)=\smallint_{[\xi]\in \widehat{G}}\textnormal{\textbf{Tr}}[\xi(x)\widehat{f}(\xi)]d\nu(\xi),\,\,f\in L^1(G)\cap L^2(G).
\end{equation*} The Plancherel theorem states that  $\mathscr{F}:L^2(G)\rightarrow \mathscr{B}^2(\widehat{G})$ is an isometry, where the inner product on $\mathscr{B}^2(\widehat{G})$ is defined via
\begin{equation}
    (f,g)_{\mathscr{B}^2(\widehat{G})}:=\smallint_{[\xi]\in \widehat{G}}\textnormal{\textbf{Tr}}[\widehat{f}(\xi)\widehat{g}(\xi)^*]d\nu(\xi).
\end{equation}So, the Plancherel theorem takes the form
\begin{equation}
    \Vert f\Vert_{L^2(G)}^2=\smallint_{[\xi]\in \widehat{G}}\|\widehat{f}(\xi)\|_{\textnormal{HS}}^2d\nu(\xi).
\end{equation}We have denoted by $\|\cdot \|_{\textnormal{HS}}$ the usual Hilbert-Schmidt norm on $\xi$-a.e.  representation space $\mathscr{B}^2(H_\xi).$

\subsection{The spaces $H^1$ and $BMO$ on Lie groups of polynomial growth}\label{BMOH1}
 Let $G$ be a Lie group of polynomial growth. In that follows we record the definitions of the Hardy space and of the $BMO$ space on a Lie group of polynomial growth ({{a Lie group of this type is unimodular,}} see \cite[Page 3]{terElstRobinsonZhu}). For this, we follow ter Elst, Robinson and Zhu \cite{terElstRobinsonZhu} and Coifman and Weiss \cite{CoifmanWeiss}. 
 
 Let us consider a sub-Laplacian $\mathcal{L}=-(X_1^2+\cdots +X_k^2)$ on $G,$ where the system of vector fields $X=\{X_i\}_{i=1}^{k}$ satisfies the H\"ormander condition of step $\kappa$. For every point $g\in G,$ let us denote  $X_{g}=\{X_{i,g}\}_{i=1}^{k},$  $\mathcal{H}_{g}=\textnormal{span}\{X_{g}\}.$ We say that a curve $\gamma:[0,1]\rightarrow G$ is  horizontal if  $$\Dot{\gamma}(t)\in \mathcal{H}_{\gamma(t)},\textnormal{   for a.e.  }t\in (0,1).$$ The Carnot-Carath\'eodory distance associated to the sub-Riemannian structure induced by $X,$ is defined by
 \begin{equation*}
     d_{s}(g_{0},g_{1}):=\inf_{\gamma \textnormal{ horizontal   }}\{l(\gamma):=\smallint\limits_{0}^{1}|\Dot{\gamma}(t)|dt:\,\,\gamma(0)=g_0,\,\gamma(1)=g_1\},\,\,g_{0},g_{1 }\in G.
 \end{equation*}
 
 We will fix a subelliptic distance on $G,$ $|\cdot|,$ defined by the Carnot-Carath\'eodory distance in the natural way: $|g|=d_s(g, e_G),$ 
where $e_{G}$ is the identity element of $G.$  As usual, the ball of radius $r>0,$ is defined as 
\begin{equation*}
    B(x,r)=\{y\in G:|y^{-1}x|<r\}.
\end{equation*}

An important geometric property of the family of balls associated to the Carnot-Carath\'eodory distance arises when one wants to compute their volume. Indeed,  there exist $d=d_0$ and $d'=d_\infty$ such that 
\begin{equation}
    \forall R\in [0,1],\, |B(x,R)|\sim R^{d_0},\,\forall R\geq 1,\, |B(x,R)|\sim R^{d_\infty},
\end{equation} where $B(x,R)$ is the ball of radius $R>0$ associated to the distance $|\cdot|.$   Then $d_0$ and $d_\infty$ are called  the local dimension and the global dimension of the pair $(G,X),$ respectively. One of the main aspects of our work, is that our kernel criteria are presented in terms of the local dimension $d_0.$ It is well known that $d_0$ depends of the H\"ormander system of vector fields $X$ and that $d_\infty$ only depends of the group $G$.

The subelliptic $BMO$ space on $G,$ $\textnormal{BMO}^{\mathcal{L}}(G),$ is the space of locally integrable functions $f$ satisfying
\begin{equation*}
    \Vert f\Vert_{\textnormal{BMO}^{\mathcal{L}}(G)}:=\sup_{\mathbb{B}}\frac{1}{|\mathbb{B}|}\smallint\limits_{\mathbb{B}}|f(x)-f_{\mathbb{B}}|dx<\infty,\textnormal{ where  } f_{\mathbb{B}}:=\frac{1}{|\mathbb{B}|}\smallint\limits_{\mathbb{B}}f(x)dx,
\end{equation*}
and $\mathbb{B}$ ranges over all balls $B(x_{0},r),$ with $(x_0,r)\in G\times (0,\infty).$ The subelliptic  Hardy space $\textnormal{H}^{1,\mathcal{L}}(G)$ will be defined via the atomic decomposition:  $f\in \textnormal{H}^{1,\mathcal{L}}(G)$ if and only if $f$ can be expressed as $$f=\sum_{j=1}^\infty c_{j}a_{j},$$ where $\{c_j\}_{j=1}^\infty$ is a sequence in $\ell^1(\mathbb{N}),$ and every function $a_j$ is an atom, i.e., $a_j$ is supported in some ball $B_j,$ ($a_j$ satisfies the cancellation property) $$\smallint\limits_{B_j}a_{j}(x)dx=0,$$ and 
\begin{equation*}
    \Vert a_j\Vert_{L^\infty(G)}\leqslant \frac{1}{|B_j|}.
\end{equation*} The norm $\Vert f\Vert_{\textnormal{H}^{1,\mathcal{L}}(G)}$ is the infimum over  all possible series $\sum_{j=1}^\infty|c_j|.$ Furthermore $\textnormal{BMO}^{\mathcal{L}}(G)$ is the topological dual of the Hardy space $\textnormal{H}^{1,\mathcal{L}}(G),$ see  Coifman and Weiss \cite{CoifmanWeiss} or ter Elst, Robinson and Zhu \cite[Page 4]{terElstRobinsonZhu}. This can be understood in the following sense:
\begin{itemize}
    \item[(a).] If $\phi\in \textnormal{BMO}^{\mathcal{L}}(G), $ then $$\Phi: f\mapsto \smallint\limits_{G}f(x)\phi(x)dx,$$ admits a bounded extension on $\textnormal{H}^{1,\mathcal{L}}(G).$
    \item[(b).] Conversely, every continuous linear functional $\Phi$ on $\textnormal{H}^{1,\mathcal{L}}(G)$ arises as in $\textnormal{(a)}$ with a unique element $\phi\in \textnormal{BMO}^{\mathcal{L}}(G).$
\end{itemize} The norm of $\phi$ as a linear functional on $\textnormal{H}^{1,\mathcal{L}}(G)$ is equivalent with the $\textnormal{BMO}^{\mathcal{L}}(G)$-norm. Important properties of the $\textnormal{BMO}^{\mathcal{L}}(G)$ and the $\textnormal{H}^{1,\mathcal{L}}(G)$ norms are the following,
\begin{equation}
 \Vert f \Vert_{\textnormal{BMO}^{\mathcal{L}}(G)}  =\sup_{\Vert g\Vert_{\textnormal{H}^{1,\mathcal{L}}(G)}=1} 
\left| \smallint\limits_{G}f(x)g(x)dx\right|,\end{equation}
\begin{equation}\Vert g \Vert_{\textnormal{H}^{1,\mathcal{L}}(G)}  =\sup_{\Vert f\Vert_{\textnormal{BMO}^{\mathcal{L}}(G)}=1} 
\left| \smallint\limits_{G}f(x)g(x)dx\right|.
\end{equation} 
If we replace $\mathcal{L}$ by the Laplacian $\mathcal{L}_G$ in the definitions above, we will write  $BMO(G)$ and $H^1(G),$ defined by the distance induced by the usual bi-invariant Riemannian metric on $G.$ The subelliptic Fefferman-Stein interpolation theorem in this case can be stated as follows (see Carbonaro, Mauceri and Meda \cite{CMM}). 
\begin{theorem}
Let $G$ be a  Lie group of polynomial growth. Let us consider a sub-Laplacian $\mathcal{L}=-(X_1^2+\cdots +X_k^2)$ on $G,$ where the system of vector fields $X=\{X_i\}_{i=1}^{k}$ satisfies the H\"ormander condition of step $\kappa$. For every $\theta\in (0,1),$ we have,
\begin{itemize}
    \item If $p_\theta=\frac{2}{1-\theta},$ then $(L^2,\textnormal{BMO}^{\mathcal{L}})_{[\theta]}(G)=L^{p_\theta}(G).$ 
    \item If $p_\theta=\frac{2}{2-\theta},$ then $(\textnormal{H}^{1,\mathcal{L}},L^2)_{[\theta]}(G)=L^{p_\theta}(G).$ 
\end{itemize}
\end{theorem}
From now, for the simplicity of the notation we will write $BMO(G)$ and $H^1(G)$ for the corresponding $BMO^{\mathcal{L}}(H)$ and $H^{1,\mathcal{L}}(G)$ associated to a sub-Laplacian $\mathcal{L}$, respectively, even when these spaces can depend of the H\"ormander system of vector-fields.

  \subsection{Homogeneous and graded Lie groups} 
    Let $G$ be a homogeneous Lie group. This means that $G$ is a connected and simply connected Lie group whose Lie algebra $\mathfrak{g}$ is endowed with a family of dilations $D_{r}^{\mathfrak{g}},$ $r>0,$ which are automorphisms on $\mathfrak{g}$  satisfying the following two conditions:
\begin{itemize}
\item For every $r>0,$ $D_{r}^{\mathfrak{g}}$ is a map of the form
$$ D_{r}^{\mathfrak{g}}=\textnormal{Exp}(\ln(r)A) $$
for some diagonalisable linear operator $A\equiv \textnormal{diag}[\nu_1,\cdots,\nu_n]$ on $\mathfrak{g}.$
\item $\forall X,Y\in \mathfrak{g}, $ and $r>0,$ $[D_{r}^{\mathfrak{g}}X, D_{r}^{\mathfrak{g}}Y]=D_{r}^{\mathfrak{g}}[X,Y].$ 
\end{itemize}
We call  the eigenvalues of $A,$ $\nu_1,\nu_2,\cdots,\nu_n,$ the dilations weights or weights of $G$.  The homogeneous dimension of a homogeneous Lie group $G$ is given by  $$ Q=\textnormal{\textbf{Tr}}(A)=\nu_1+\cdots+\nu_n.  $$
The dilations $D_{r}^{\mathfrak{g}}$ of the Lie algebra $\mathfrak{g}$ induce a family of  maps on $G$ defined via,
$$ D_{r}:=\exp_{G}\circ D_{r}^{\mathfrak{g}} \circ \exp_{G}^{-1},\,\, r>0, $$
where $\exp_{G}:\mathfrak{g}\rightarrow G$ is the usual exponential mapping associated to the Lie group $G.$ We refer to the family $D_{r},$ $r>0,$ as dilations on the group. If we write $rx=D_{r}(x),$ $x\in G,$ $r>0,$ then a relation on the homogeneous structure of $G$ and the Haar measure $dx$ on $G$ is given by $$ \smallint\limits_{G}(f\circ D_{r})(x)dx=r^{-Q}\smallint\limits_{G}f(x)dx. $$
    
A  Lie group is graded if its Lie algebra $\mathfrak{g}$ may be decomposed as the sum of subspaces $\mathfrak{g}=\mathfrak{g}_{1}\oplus\mathfrak{g}_{2}\oplus \cdots \oplus \mathfrak{g}_{s}$ such that $[\mathfrak{g}_{i},\mathfrak{g}_{j} ]\subset \mathfrak{g}_{i+j},$ and $ \mathfrak{g}_{i+j}=\{0\}$ if $i+j>s.$  Examples of such groups are the Heisenberg group $\mathbb{H}^n$ and more generally any stratified groups where the Lie algebra $ \mathfrak{g}$ is generated by $\mathfrak{g}_{1}$.  Here, $n$ is the topological dimension of $G,$ $n=n_{1}+\cdots +n_{s},$ where $n_{k}=\mbox{dim}\mathfrak{g}_{k}.$

A Lie algebra admitting a family of dilations is nilpotent, and hence so is its associated
connected, simply connected Lie group.  

A graded Lie group $G$ is a homogeneous Lie group equipped with a family of weights $\nu_j,$ all of them positive rational numbers. Let us observe that if $\nu_{i}=\frac{a_i}{b_i}$ with $a_i,b_i$ integer numbers,  and $b$ is the least common multiple of the $b_i's,$ the family of dilations 
$$ \mathbb{D}_{r}^{\mathfrak{g}}=\textnormal{Exp}(\ln(r^b)A):\mathfrak{g}\rightarrow\mathfrak{g}, $$
have integer weights,  $\nu_{i}=\frac{a_i b}{b_i}. $ So, in this paper we always assume that the weights $\nu_j,$ defining the family of dilations are non-negative integer numbers which allow us to assume that the homogeneous dimension $Q$ is a non-negative integer number. This is a natural context for the study of Rockland operators (see Remark 4.1.4 of \cite{FischerRuzhanskyBook}).

\subsection{The spaces $H^1$ and $BMO$ on homogeneous Lie groups}
We will fix a homogeneous quasi-norm on $G,$ $|\cdot|.$ This means that $|\cdot|$ is a non-negative function on $G,$ satisfying 
\begin{equation}\label{symmetric:prop}
    |x|=|x^{-1}|,\,\,\,r|x|=|D_r( x)|,\,\,\,\textnormal{ and }|x|=0 \textnormal{ if and only if  }x=e_{G},
\end{equation}
where $e_{G}$ is the identity element of $G.$ It satisfies a triangle
inequality with a constant: there exists a constant $\gamma\geq 1$ such that $|xy|\leq \gamma(|x|+|y|).$ As usual, the ball of radius $r>0,$ is defined as 
\begin{equation*}
    B(x,r)=\{y\in G:|y^{-1}x|<r\}.
\end{equation*}Then $BMO(G)$ is the space of locally integrable functions $f$ satisfying
\begin{equation*}
    \Vert f\Vert_{BMO(G)}:=\sup_{\mathbb{B}}\frac{1}{|\mathbb{B}|}\smallint\limits_{\mathbb{B}}|f(x)-f_{\mathbb{B}}|dx<\infty,\textnormal{ where  } f_{\mathbb{B}}:=\frac{1}{|\mathbb{B}|}\smallint\limits_{\mathbb{B}}f(x)dx,
\end{equation*}
and $\mathbb{B}$ ranges over all balls $B(x_{0},r),$ with $(x_0,r)\in G\times (0,\infty).$ The Hardy space $H^1(G)$ will be defined via the atomic decomposition. Indeed, $f\in H^1(G),$ if and only if, $f$ can be expressed as $f=\sum_{j=1}^\infty c_{j}a_{j},$ where $\{c_j\}_{j=1}^\infty$ is a sequence in $\ell^1(\mathbb{N}_0),$ and every function $a_j$ is an atom, i.e., $a_j$ is supported in some ball $B_j,$ $\smallint_{B_j}a_{j}(x)dx=0,$ and 
\begin{equation*}
    \Vert a_j\Vert_{L^\infty(G)}\leq \frac{1}{|B_j|}.
\end{equation*} The norm $\Vert f\Vert_{H^1(G)}$ is the infimum over  all possible series $\sum_{j=1}^\infty|c_j|.$ Furthermore $BMO(G)$ is the dual of $H^1(G),$ (see Folland and Stein \cite{FollandStein1982}). This can be understood in the following sense:
\begin{itemize}
    \item[(a).] If $\phi\in BMO(G), $ then $\Phi: f\mapsto \smallint\limits_{G}f(x)\phi(x)dx,$ admits a bounded extension on $H^1(G).$
    \item[(b).] Conversely, every continuous linear functional $\Phi$ on $H^1(G)$ arises as in $\textnormal{(a)}$ with a unique element $\phi\in BMO(G).$
\end{itemize} The norm of $\phi$ as a linear functional on $H^1(G)$ is equivalent with the $BMO(G)$-norm. Important properties of the $BMO(G)$ and the $H^1(G)$ norms are the following,
\begin{equation}\label{BMOnormduality}
 \Vert f \Vert_{BMO(G)}  =\sup_{\Vert g\Vert_{H^1(G)}=1} 
\left| \smallint\limits_{G}f(x)g(x)dx\right|,\end{equation}
\begin{equation}\label{BMOnormduality'}\Vert g \Vert_{H^1(G)}  =\sup_{\Vert f\Vert_{BMO(G)}=1} 
\left| \smallint\limits_{G}f(x)g(x)dx\right|.
\end{equation}

\section{Proof of the main theorem}\label{ProofMainTh}

Before presenting the proof of our main result (Theorem \ref{main:th:2}) we present some auxiliary facts to be used in our further analysis.
\begin{remark}\label{remark31}
For the proof of the boundedness Theorem \ref{main:th:2} let us use  the equivalence of norms (whose proof follows the lines of its Euclidean analogue, see Duoandikoetxea \cite[Proposition 6.15, Page 117]{Duoandikoetxea2000})
\begin{equation}\label{eq:bmo:a}
  \Vert f \Vert_{BMO}\sim  \sup_{\delta>0,\,x_0\in G}\inf_{a\in \mathbb{C}}   \frac{1}{|B(x_0,\delta)|}\smallint\limits_{B(x_0,\delta)}|f(x)-a|dx,
\end{equation}  for any $f\in BMO(G).$ Indeed, by the definition of the $BMO$-norm, 
\begin{align*}
    \Vert f \Vert_{BMO}=\sup_{\delta>0,\,x_0\in G}\frac{1}{|B(x_0,\delta)|}\smallint\limits_{B(x_0,\delta)}|f(x)-f_{B}|dx,\,f_{B}:=\frac{1}{|B(x_0,\delta)|}\smallint f(x)dx
\end{align*}
it is clear that 
\begin{align*}
    \sup_{\delta>0}\inf_{a\in \mathbb{C}}   \frac{1}{|B(x_0,\delta)|}\smallint\limits_{B(x_0,\delta)}|f(x)-a|dx\leq \Vert f \Vert_{BMO}.
\end{align*}On the other hand, note that for any $a\in \mathbb{C},$ and all balls $B=B(x_0,\delta),$ the triangle inequality gives the estimate
\begin{align*}
    \smallint\limits_B|f(x)-f_B|dx\leq  \smallint\limits_B|f(x)-a|dx+ \smallint\limits_B|a-f_B|dx\leq 2 \smallint\limits_B|f(x)-a|dx.
\end{align*}Now, by dividing both sides by the Haar measure of $B,$ by taking the infimum over $a\in \mathbb{C}$ and the supremun over $B$ we conclude that
\begin{align*}
    \Vert f\Vert_{BMO}\leq 2 \sup_{\delta>0,\,x_0\in G}\inf_{a\in \mathbb{C}}   \frac{1}{|B(x_0,\delta)|}\smallint\limits_{B(x_0,\delta)}|f(x)-a|dx,
\end{align*}as desired.

\end{remark}

\begin{remark}\label{remark32}
We record that if $G$ is a graded Lie group, a quasi-distance on $G$ is symmetric (that is, $|x|=|x^{-1}|$ for any $x\in G$), see \eqref{symmetric:prop} or  \cite[Page 109]{FischerRuzhanskyBook}.  On the other hand, the Carnot-Carath\'eodory distance on a Lie group of polynomial growth $G$  is also symmetric  in view of its $G$-invariance (see Coulhon, Saloff-Coste and Varopoulos \cite[Page 40]{Saloff-CosteBook}).   

\end{remark}

When estimating the action of singular integrals  on functions supported in small balls we require suitable versions of Sobolev inequalities. We record it in the following  remarks.
\begin{remark}[Sobolev inequality on unimodular Lie groups]\label{Sob:comp} Let $G$ be a unimodular Lie group and let $\mathcal{L}$ be a positive sub-Laplacian associated to a H\"ormander system of vector fields $X=\{X_i\}_{i=1}^k$. For $0< p<\infty,$ the subelliptic $L^p$-Sobolev space of order $s\in\mathbb{R},$ is defined by the family of distributions $f\in \mathscr{D}'(G)$ such that  
\begin{equation*}
     \Vert f \Vert_{ L^{p,\mathcal{L}}_{s}(G) }:=\Vert  (1+\mathcal{L})^{\frac{s}{2}} f \Vert_{ L^{p}(G) }<\infty.
 \end{equation*}In our further analysis we will make use of the following Sobolev inequality (see e.g. Coulhon, Saloff-Coste and Varopoulos \cite{Saloff-CosteBook})
\begin{equation}\label{SI:unimodular}
    \Vert f \Vert_{L^p(G)}\leq C\Vert (1+\mathcal{L})^{\frac{a}{2}}f \Vert_{L^q(G)},\quad a=d_0\left(\frac{1}{q}-\frac{1}{p}\right),\,1<q<p<\infty,
\end{equation}where $d_0$ is the local dimension of $(G,X).$ Note that \eqref{SI:unimodular} covers the case where $G$ is of polynomial growth.
\end{remark}
\begin{remark}[Sobolev inequality on graded Lie groups]\label{Sob:graded}
Let $G$ be a graded Lie group and let $\mathcal{R}$ be a positive Rockland operator of homogeneous degree $\nu>0.$ For $0< p<\infty,$ the  $L^p$-Sobolev space of order $s\in\mathbb{R},$ on $G$ is defined by the family of distributions $f\in \mathscr{D}'(G)$ such that  
\begin{equation*}
     \Vert f \Vert_{ L^{p}_{s}(G) }:=\Vert  (1+\mathcal{R})^{\frac{s}{\nu}} f \Vert_{ L^{p}(G) }<\infty.
 \end{equation*}In our further analysis we will make use of the following Sobolev inequality on graded groups (see \cite[Page 1674]{FischerRuzhansky2017})
\begin{equation}
    \Vert f \Vert_{L^p(G)}\leq C\Vert (1+\mathcal{R})^{\frac{a}{\nu}} f\Vert_{L^q(G)},\quad a=Q\left(\frac{1}{q}-\frac{1}{p}\right),\,1<q<p<\infty,
\end{equation}where $Q$ is the homogeneous of the group.
\end{remark}

\begin{remark}
Note that the condition
    \eqref{Fourier:growth:1} and the Plancherel theorem   imply the $L^2$-boundedness of the operator $(1+\mathcal{R})^{\frac{Q\theta}{2\nu}}T,$  since   in terms of the group Fourier transform we have that
 \begin{equation*}
   (1+\mathcal{R})^{\frac{Q\theta}{2\nu}}Tf(x)=\smallint_{\widehat{G}}\textnormal{Tr}[\pi(x)(1+\pi(\mathcal{R}))^{\frac{Q\theta}{2\nu}}\widehat{K}(\pi)\widehat{f}(\pi)]d\pi,\,f\in C^{\infty}_0(G). 
\end{equation*} Note that in the previous identity we have used that for any $s\in \mathbb{R},$ $$\mathscr{F}[(1+\mathcal{R})^{\frac{s}{\nu}}\delta_e])(\pi)=(1+\pi({\mathcal{R}}))^{\frac{s}{\nu}},$$ in view of the functional calculus for Rockland operators, see \cite[Chapter IV]{FischerRuzhanskyBook}. Here $\delta_e$ denotes the delta distribution at the identity $e\in G.$   Observe that we have written $d\pi=d\nu(\pi)$ for the Plancherel measure on $\widehat{G}.$ Similarly, the condition in \eqref{Fourier:growth:1New} is equivalent to the $L^2$-boundedness of the operator $T(1+\mathcal{R})^{\frac{Q\theta}{2\nu}}.$
\end{remark}

\begin{proof}[Proof of Theorem \ref{main:th:2}]
Let $e\in G$ be the identity element of $G$, and for any $z\in G,$ let $B(z,\delta)$ be the ball of radius $\delta>0$ centred at $z.$ For simplicity for $z=e$ we will denote $$B(\delta):=B(e,\delta),\,\quad |B(\delta)|:=\textnormal{Vol}(B(\delta)).$$

We want to prove the boundedness inequality
\begin{equation}\label{ineq:delta:1}
 \sup_{\delta>0}\inf_{a\in \mathbb{C}}   \frac{1}{|B(x_0,\delta)|}\smallint\limits_{B(x_0,\delta)}|Tf(x)-a|dx\leq C\Vert f\Vert_{L^\infty(G)}.
\end{equation}
The equivalence in \eqref{eq:bmo:a}  provides 
a  way to show that $Tf\in BMO$ without using its average on $B(x_0,\delta)$: it suffices  to 
find  a  constant $a$  (that can depend on $B(x_0,\delta)$).

The first step of the proof is to reduce the estimate to the case where $x_0=e,$ that is
\begin{equation}\label{ineq:delta:2}
 \sup_{\delta>0}\inf_{a\in \mathbb{C}}   \frac{1}{|B(\delta)|}\smallint\limits_{B(\delta)}|Tf(x)-a|dx\leq C\Vert f\Vert_{L^\infty(G)}.
\end{equation}
For this, let us use the left-invariance of $T.$ Indeed, observe that the identities
\begin{align*}
     \sup_{\delta>0}\inf_{a\in \mathbb{C}}   \frac{1}{|B(x_0,\delta)|}\smallint\limits_{B(x_0,\delta)}|Tf(x)-a|dx &=\sup_{\delta>0}\inf_{a\in \mathbb{C}}   \frac{1}{|B(\delta)|}\smallint\limits_{B(\delta)}|Tf(x_0z)-a|dz\\
     &=\sup_{\delta>0}\inf_{a\in \mathbb{C}}   \frac{1}{|B(\delta)|}\smallint\limits_{B(\delta)}|T(f(x_0\cdot))(z)-a|dz
\end{align*}prove the equivalence between \eqref{ineq:delta:1} and \eqref{ineq:delta:2}. So, in order to prove \eqref{ineq:delta:2} let us fix  a function $f\in L^\infty(G),$ and let us consider the decomposition
$$  f=f_1+f_2, \quad \, f_1=f\cdot 1_{ \{x\in G: |x|\leq 2\delta^{1-\theta}  \}},\quad f_2=f\cdot 1_{ \{x\in G: |x|> 2\delta^{1-\theta}  \}}, $$ in the case where $0<\delta<1.$ This is the relevant case. Indeed,   if $\delta>1,$ we take $c>0$ small enough so that $c\delta^{1-\theta}<\delta+1.$ Then, define $f_1(x)=f(x)1_{\{x\in G:|x|\leq c\delta^{1-\theta}\} }.$ Note that $Tf_2=0,$ where $f_2=f-f_1.$ The $L^2$-boundedness of $T$ implies that
$$  \frac{1}{|B(\delta)|}|Tf(x)|dx=\frac{1}{|B(\delta)|}|Tf_1(x)|dx\lesssim\delta^{-\frac{Q}{2}} \Vert Tf_1\Vert_{L^2}\lesssim \delta^{-\frac{Q}{2}}\delta^{Q/2}\Vert f\Vert_{L^\infty}. $$
Then we have that
\begin{equation}\label{deltagrandequeuno}
 \sup_{\delta>1}\inf_{a\in \mathbb{C}}   \frac{1}{|B(x_0,\delta)|}\smallint\limits_{B(x_0,\delta)}|Tf(x)-a|dx\leq C\Vert f\Vert_{L^\infty(G)}.
\end{equation}Note that a similar analysis can be applied to the adjoint $T^*$ of $T$ in order to get the estimate 
\begin{equation}\label{deltagrandequeuno2}
 \sup_{\delta>1}\inf_{a\in \mathbb{C}}   \frac{1}{|B(x_0,\delta)|}\smallint\limits_{B(x_0,\delta)}|T^*f(x)-a|dx\leq C\Vert f\Vert_{L^\infty(G)}.
\end{equation}
Indeed, by the duality argument the $L^2$-boundedness of $T$ and $T^*$ are equivalent facts.

We illustrate the decomposition for the case $\delta\in (0,1)$  of the domain of $f$ in the case of the $2D$-torus $G=\mathbb{T}^2$  in  Figure \ref{Figure 1}.
\begin{figure}[h]
\includegraphics[width=6cm]{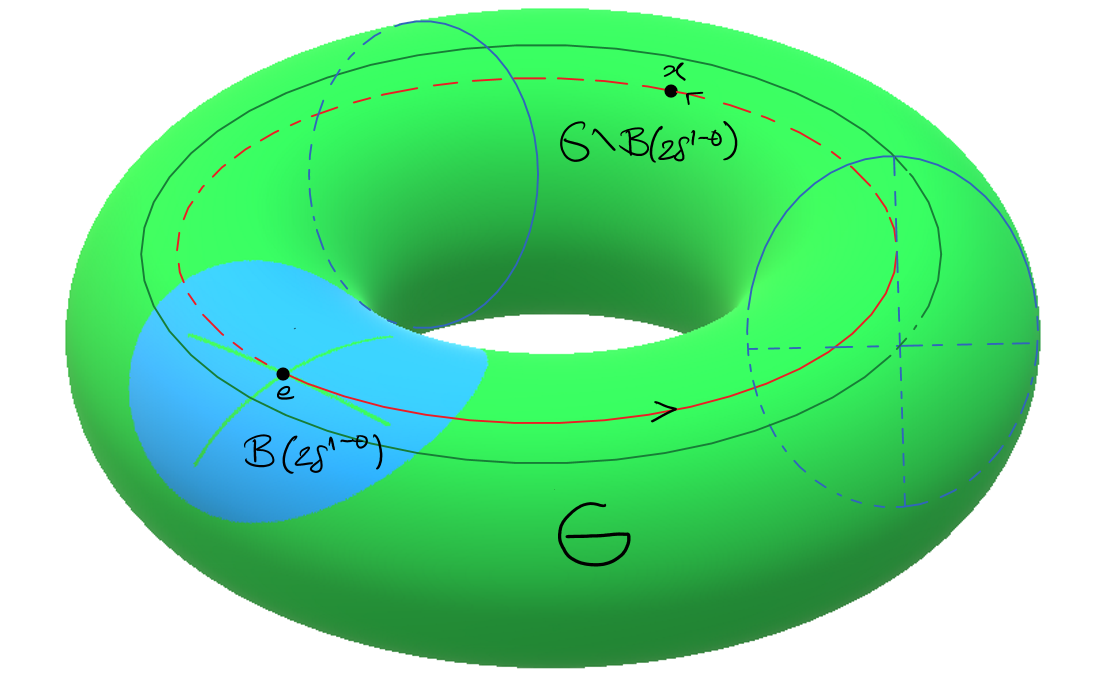}\\
\caption{}
\label{Figure 1}
\centering
\end{figure}
So, for any $0<\delta<1$ we are going to prove the estimate
\begin{equation}\label{f1estimate}
     \frac{1}{|B(\delta)|}\smallint\limits_{B(\delta)}|Tf_1(x)|dx
    \lesssim \Vert f\Vert_{L^\infty(G)},
\end{equation}
if one of the conditions \eqref{Fouriergrowth:2} or \eqref{Fourier:growth:1} hold true,  and the estimate 
\begin{equation}\label{New:adjoint}
    \frac{1}{|B(\delta)|}\smallint\limits_{B(\delta)}|T^{*}f_1(x)|dx
    \lesssim \Vert f\Vert_{L^\infty(G)},
\end{equation} in the case where one of the conditions  \eqref{Fouriergrowth:2New2} or \eqref{Fourier:growth:1New} remain valid, with $T^*$ being the $L^2$-adjoint of $T.$ For this we will analyse separately two cases: the first one, is the case where $G$ is a Lie group of polynomial growth endowed with a H\"ormander system of vector fields $X=\{X_{i}\}_{i=1}^k,$ while in the second one  $G$ is a graded Lie group of homogeneous dimension $Q=\nu_1+\cdots+ \nu_n$  where a Rockland operator $\mathcal{R}$ is fixed. 
\begin{itemize}
    \item Case 1.  Let us assume that $G$ is a Lie group of of polynomial growth and let $$ \mathcal{L}=-\sum_{i=1}^kX_i^2$$ be a positive sub-Laplacian associated  with a H\"ormander system of vector fields $X=\{X_{i}\}_{i=1}^k$. Let $d=d_0$ be the local dimension of $(G,X).$
    \begin{itemize}
        \item Case 1a. Assume that $(1+\mathcal{L})^{\frac{d_0\theta}{4}} T$ is bounded on $L^2.$  Using the  Sobolev inequality (see Remark \ref{Sob:comp}) we have
\begin{align*}
  \Vert Tf_1 \Vert_{L^p(G)}= \Vert (1+\mathcal{L})^{-\frac{d_0\theta}{4}}(1+\mathcal{L})^{\frac{d_0\theta}{4}} Tf_1 \Vert_{L^p(G)}\leq C\Vert (1+\mathcal{L})^{\frac{d_0\theta}{4}} Tf_1 \Vert_{L^2(G)}   
\end{align*}
with $\frac{1}{p}=\frac{1-\theta}{2}.$ Now, using that the operator $(1+\mathcal{L})^{\frac{d_0\theta}{4}} T$ is bounded on $L^2(G),$ and that the group $G$ satisfies the global doubling condition, we have that
\begin{align*}
  \Vert Tf_1 \Vert_{L^p(G)} &\lesssim \Vert f_1 \Vert_{L^2(G)}\leq \Vert f \Vert_{L^\infty(G)}|B(e,2\delta^{1-\theta})|^{\frac{1}{2}}\lesssim\Vert f \Vert_{L^\infty(G)}|B(\delta^{1-\theta})|^{\frac{1}{2}}\\
  &\lesssim\Vert f \Vert_{L^\infty(G)}|B(\delta)|^{\frac{1-\theta}{2}}.
\end{align*}In consequence
\begin{align*}
    \frac{1}{|B(\delta)|}\smallint\limits_{B(\delta)}|Tf_1(x)|dx
    &\lesssim\frac{1}{|B(\delta)|}\left(\smallint\limits_{B(\delta)}|Tf_1(x)|^pdx\right)^{\frac{1}{p}}|B(\delta)|^{\frac{1}{p'}}\\
    &\leq|B(\delta)|^{-1}\Vert Tf_{1}\Vert_{L^p(G)}|B(\delta)|^{\frac{(p-1)}{p}}.
\end{align*}
Observing that $\frac{p-1}{p}=1-\frac{1}{p}=\frac{1+\theta}{2},$ we have
\begin{align*}
  \frac{1}{|B(\delta)|}\smallint\limits_{B(\delta)}|Tf_1(x)|dx
    &\lesssim|B(\delta)|^{-1}|B(\delta)|^{\frac{(1-\theta)}{2}} |B(\delta)|^{\frac{(p-1)}{p}} \Vert f \Vert_{L^\infty(G)}\\
    &\lesssim|B(\delta)|^{ -1+\frac{(1-\theta)}{2} +\frac{(1+\theta)}{2} }\Vert f \Vert_{L^\infty(G)}=\Vert f \Vert_{L^\infty(G)}. 
\end{align*}

    \item Case 1b. Now, assume that $T(1+\mathcal{L})^{\frac{d_0\theta}{4}}$ is bounded on $L^2(G).$ Observe also that, by following the argument in Case 1a above we also can prove the following estimate for the adjoint operator $T^{*}$ of $T,$
\begin{equation}\label{second:cond: ope}
  \frac{1}{|B(\delta)|}\smallint\limits_{B(\delta)}|T^{*}f_1(x)|dx\lesssim \Vert f\Vert_{L^\infty(G)},  
\end{equation}   indeed, such a estimate is independent of the kernel conditions. We start by observing that
\begin{align*}
    \Vert T^{*}f_1\Vert_{L^p(G)} = \Vert (1+\mathcal{L})^{-\frac{d_0\theta}{4}}(1+\mathcal{L})^{\frac{d_0\theta}{4}} T^*f_1 \Vert_{L^p(G)}\leq C\Vert (1+\mathcal{L})^{\frac{d_0\theta}{4}} T^*f_1 \Vert_{L^2(G)},
\end{align*} in view of the Sobolev inequality.
Note that we have the identity of operators
$$ (1+\mathcal{L})^{\frac{d_0\theta}{4}} T^*=[T(1+\mathcal{L})^{\frac{d_0\theta}{4}} ]^* . $$
Since
$T(1+\mathcal{L})^{\frac{d_0\theta}{4}} $ is bounded on $L^2(G),$ we also have that its adjoint $(1+\mathcal{L})^{\frac{d_0\theta}{4}}T^*$ is bounded on $L^2(G).$ In consequence we have the estimate
\begin{align*}
    \Vert (1+\mathcal{L})^{\frac{d_0\theta}{4}} T^*f_1 \Vert_{L^2(G)}\leq C\Vert f_1\Vert_{L^2(G)}.
\end{align*}So, we can proceed as in Case 1a in order to get the estimate in \eqref{second:cond: ope}.
\end{itemize}

\item Case 2. Let us assume that $G$ is a graded Lie group and let $\mathcal{R}$ be a positive Rockland operator of homogeneous degree $\nu>0.$ Using the same approach like in Case 1, we can estimate
\begin{align*}
  \Vert Tf_1 \Vert_{L^p(G)}= \Vert (1+\mathcal{R})^{-\frac{Q\theta}{2\nu}}(1+\mathcal{R})^{\frac{Q\theta}{2\nu}} Tf_1 \Vert_{L^p(G)}\leq C\Vert (1+\mathcal{R})^{\frac{Q\theta}{2\nu}} Tf_1 \Vert_{L^2(G)},   \, 
\end{align*}where we have used the Sobolev inequality on graded groups (see Remark \ref{Sob:graded}), with $\frac{1}{p}=\frac{1-\theta}{2}.$ Now, if we use the $L^2$-boundedness of the operator $(1+\mathcal{R})^{\frac{Q\theta}{2\nu}} T,$ we can proceed as in Case 1 in order to get \eqref{f1estimate} by observing that for all $\delta>0,$ $|B(\delta)|\sim \delta^{Q}$. Also, if on the other hand, we have the $L^2$-boundedness of the operator  $T(1+\mathcal{R})^{\frac{Q\theta}{2\nu}} $ we can conclude the estimate in \eqref{New:adjoint}, by observing again that for all $\delta>0,$ $|B(\delta)|\sim \delta^{Q},$ and the estimate
\begin{align*}
  \Vert T^*f_1 \Vert_{L^p(G)} &= \Vert (1+\mathcal{R})^{-\frac{Q\theta}{2\nu}}(1+\mathcal{R})^{\frac{Q\theta}{2\nu}} T^*f_1 \Vert_{L^p(G)}\leq C\Vert [T(1+\mathcal{R})^{\frac{Q\theta}{2\nu}}]^* f_1 \Vert_{L^2(G)}\\
  \,& \leq C\Vert [T(1+\mathcal{R})^{\frac{Q\theta}{2\nu}}]^*\Vert_{\mathcal{B}(L^2)} \Vert f_1 \Vert_{L^2(G)}\\
  & =C\Vert T(1+\mathcal{R})^{\frac{Q\theta}{2\nu}}\Vert_{\mathcal{B}(L^2)}\Vert f_1 \Vert_{L^2(G)}\lesssim\Vert f_1\Vert_{L^2(G)}.
\end{align*}
\end{itemize}

Now, let us apply the $BMO$-property in Remark \ref{remark31}. Again, we will consider two cases.

\begin{itemize}
    \item Case A. Let us assume that the kernel of  $T$ satisfies the estimate
 $$  [K]_{H_{\infty,\theta}}:=\sup_{0<R<1}\frac{1}{|B(R)| }\smallint\limits_{|y|<R} \smallint\limits_{|x|\geq 2R^{1-\theta}}|K(xy^{-1})-K(x)|dx   <\infty,$$ and one of the conditions \eqref{Fouriergrowth:2} or \eqref{Fourier:growth:1} hold true.  
 Define $$a_\delta:=\smallint\limits_{G}K(y^{-1})f_{2}(y)dy.$$ Then,
\begin{align*}
    |Tf_{2}(x)-a_\delta|\leq \smallint\limits_{|y|\geq 2\delta^{1-\theta}}|K(y^{-1}x)-K(y^{-1})||f_2(y)|dy=\smallint\limits_{|y^{-1}|\geq 2\delta^{1-\theta}}|K(yx)-K(y)||f(y^{-1})|dy.
\end{align*}
Now, using Remark \ref{remark32} we have the property $|x|=|x^{-1}|$ for any $x\in G$  and consequently
\begin{align*}
  &\frac{1}{|B(\delta)|}\smallint_{B(\delta)}  |Tf_{2}(x)-a_\delta|dx\\ &\lesssim\frac{1}{|B(\delta)|}\smallint\limits_{|x|\leq \delta}\smallint\limits_{|y^{-1}|\geq 2\delta^{1-\theta}}|K(yx)-K(y)||f(y^{-1})|dy\,dx\\
  &\leq \frac{1}{|B(\delta)|}\smallint\limits_{|x^{-1}|\leq \delta}\smallint\limits_{|y^{-1}|\geq 2\delta^{1-\theta}}|K(yx^{-1})-K(y)|dydx\Vert f \Vert_{L^\infty(G)}\\
  &= \frac{1}{|B(\delta)|}\smallint\limits_{|x|\leq \delta}\smallint\limits_{|y|\geq 2\delta^{1-\theta}}|K(yx^{-1})-K(y)|dydx\Vert f \Vert_{L^\infty(G)}\\
   &= \frac{1}{|B(\delta)|}\smallint\limits_{|y|\leq \delta}\smallint\limits_{|x|\geq 2\delta^{1-\theta}}|K(xy^{-1})-K(x)|dydx\Vert f \Vert_{L^\infty(G)}\lesssim [K]_{\infty,\theta}\Vert f \Vert_{L^\infty(G)}.
\end{align*}
So, we can deduce the estimates
\begin{align*}
  \frac{1}{|B(\delta)|}\smallint_{B(\delta)}  |Tf(x)-a_\delta|dx&\leq  \frac{1}{|B(\delta)|}\smallint\limits_{B(\delta)}|Tf_1(x)|dx\\
  &+ \frac{1}{|B(\delta)|}\smallint_{B(\delta)}  |Tf_{2}(x)-a_\delta|dx\lesssim \Vert f\Vert_{L^\infty},  
\end{align*} for any $\delta\in (0,1).$ Then,  by combining the previous analysis with \eqref{deltagrandequeuno} we get the inequality
\begin{align*}
\Vert Tf\Vert_{BMO}\sim \sup_{\delta>0}   \inf_{a\in \mathbb{C}}\frac{1}{|B(\delta)|}\smallint_{B(\delta)}  |Tf(x)-a|dx\lesssim\Vert f\Vert_{L^\infty(G)}.
\end{align*}
\item Case B. Now, let us consider the case where $T$ satisfies the kernel condition 
$$  [K]_{H_{\infty,\theta}}':=\sup_{0<R<1}\frac{1}{|B(R)|  }\smallint\limits_{|y|<R} \smallint\limits_{|x|\geq 2R^{1-\theta}}|K(y^{-1}x)-K(x)|dx  <\infty,$$ and when one of the conditions  \eqref{Fouriergrowth:2New2} or \eqref{Fourier:growth:1New} remain valid. We have the estimate in \eqref{New:adjoint}, that is 
$$ \frac{1}{|B(\delta)|}\smallint\limits_{B(\delta)}|T^{*}f_1(x)|dx
    \lesssim \Vert f\Vert_{L^\infty(G)}. $$
Also, by replacing the previous analysis with $T^{*}$ instead of $T,$ we have the inequality
\begin{align*}
  &\frac{1}{|B(\delta)|}\smallint_{B(\delta)}  |T^*f_{2}(x)-a_\delta|dx\\ &\lesssim\frac{1}{|B(\delta)|}\smallint\limits_{|y|\leq \delta}\smallint\limits_{|x|\geq 2\delta^{1-\theta}}|K^{*}(xy^{-1})-K^{*}(x)|dydx\Vert f \Vert_{L^\infty(G)},
\end{align*}where $K^*$ is the right-convolution kernel associated to $T^*.$ Using the identity of distributions $K^{*}(x)=\overline{K}(x^{-1}),$ $x\in G,$ we get
\begin{align*}
    \frac{1}{|B(\delta)|}\smallint\limits_{|y|\leq \delta}\smallint\limits_{|x|\geq 2\delta^{1-\theta}}|K^{*}(xy^{-1})-K^{*}(x)|dydx &=\frac{1}{|B(\delta)|}\smallint\limits_{|y|\leq \delta}\smallint\limits_{|x|\geq 2\delta^{1-\theta}}|K(yx^{-1})-K(x^{-1})|dydx\\
    &=\frac{1}{|B(\delta)|}\smallint\limits_{|y^{-1}|\leq \delta}\smallint\limits_{|x^{-1}|\geq 2\delta^{1-\theta}}|K(y^{-1}x)-K(x)|dydx\\
     &=\frac{1}{|B(\delta)|}\smallint\limits_{|y|\leq \delta}\smallint\limits_{|x|\geq 2\delta^{1-\theta}}|K(y^{-1}x)-K(x)|dydx\\
     &\lesssim [K]_{H_{\infty,\theta}}'.
\end{align*}So, we can deduce the estimate
\begin{align*}
  \frac{1}{|B(\delta)|}\smallint_{B(\delta)}  |T^*f(x)-a_\delta|dx&\leq  \frac{1}{|B(\delta)|}\smallint\limits_{B(\delta)}|T^*f_1(x)|dx\\
  &+ \frac{1}{|B(\delta)|}\smallint_{B(\delta)}  |T^*f_{2}(x)-a_\delta|dx\lesssim \Vert f\Vert_{L^\infty}. 
\end{align*}Using the analysis above for $0<\delta<1$ and the inequality in  \eqref{deltagrandequeuno2} we get
\begin{align*}
\Vert T^*f\Vert_{BMO}\sim \sup_{\delta>0}   \inf_{a\in \mathbb{C}}\frac{1}{|B(\delta)|}\smallint_{B(\delta)}  |T^{*}f(x)-a|dx\lesssim\Vert f\Vert_{L^\infty(G)}.
\end{align*}So, we have proved the boundedness of $T^*$ from $L^\infty(G)$ into $BMO(G),$ and hence, by the duality argument, we deduce that $T$ is bounded from $H^{1}(G)$ into $L^1(G).$ 
\end{itemize}The proof of Theorem \ref{main:th:2} is complete. 
\end{proof}

\section{Examples}\label{Examples} We end this paper by illustrating our main Theorem \ref{main:th:2} with some (possibly well known) examples (with the exception of Subsection \ref{LGPGTI}). Applications of Theorem \ref{main:th:2} in the context of graded Lie groups can be found in the second part of this work, namely,  in \cite{BjorkSjolin}. We start with the case of a compact Lie group endowed with its standard  Riemannian structure. 

\subsection{Riemannian structures on compact Lie groups}\label{example:G:riem}
Let $G$ be a compact Lie group. Assume $G$ to be endowed with its unique (up to a constant) bi-invariant Riemannian metric. Let $\mathcal{L}_G$ be the positive Laplace-Beltrami operator on $G.$ For an orthonormal basis $\{X_i\}_{i=1}^n,$ $n=\dim(G),$ with respect to the aforementioned  Riemannian metric, we have
\begin{equation*}
    \mathcal{L}_G=-(X_1^2+\cdots +X_n^2).
\end{equation*}If we replace the sub-Laplacian $\mathcal{L}$ by the Laplace-Beltrami operator $\mathcal{L}_G,$ the corresponding Carnot-Carath\'eodory distance associated with $\mathcal{L}$ can be replaced by the standard geodesic distance on $G,$ and in this context $d_0$ in Theorem \ref{main:th:2} is the topological dimension $n.$

The $L^2$-spectrum of the Laplace-Beltrami operator $\textnormal{Spect}(\mathcal{L}_G)=\{\lambda_{[\xi]}\}_{[\xi]\in \widehat{G}}$ can be enumerated by using the unitary dual of $G$, and the Fourier transform of its right-convolution kernel  takes the form
\begin{equation}\label{LGriem}
    \widehat{\mathcal{L}}_{G}(\xi)=\lambda_{[\xi]}I_{H_\xi}, \,[\xi]\in \widehat{G},
\end{equation}where $I_{H_\xi}$ is the identity operator on any representation space $H_\xi.$ The condition in \eqref{Fouriergrowth:2} is equivalent to the following estimate
\begin{equation}\label{eq:Ft:G}
    \sup_{[\xi]\in \widehat{G}}\Vert \langle \xi\rangle^{\frac{n\theta}{2}}\widehat{K}(\xi) \Vert_{\textnormal{op}}<\infty,\,\,\ \langle \xi\rangle:=(1+\lambda_{[\xi]})^{\frac{1}{2}}.
\end{equation}Observe that \eqref{eq:Ft:G} is equivalent to the 
existence of a finite constant $C>0,$ such that
\begin{equation}\label{LxiG}
  \forall [\xi]\in \widehat{G},\,\,  \Vert \widehat{K}(\xi)\Vert_{\textnormal{op}}\leq C\langle \xi\rangle^{-\frac{n\theta}{2}},\,\ \langle \xi\rangle:=(1+\lambda_{[\xi]})^{\frac{1}{2}}.
\end{equation} In view of Theorem \ref{main:th:2}, a convolution operator $T$ (defined via $Tf=f\ast K,$ $f\in C^\infty(G)$) associated to a right-convolution kernel $K,$ satisfying  \eqref{LxiG} and the kernel condition
\begin{equation}\label{GS:CZ:cond:2':2}
        [K]_{H_{\infty,\theta}}:=\sup_{0<R<1}\sup_{|y|<R} \smallint\limits_{|x|\geq 2R^{1-\theta}}|K(xy^{-1})-K(x)|dx   <\infty,
    \end{equation}  admits a bounded extension from  $L^\infty(G) $ into $ BMO(G).$ Moreover, if $K$ satisfies \eqref{LxiG} and the kernel condition
    \begin{equation}\label{GS:CZ:cond:22':2}
        [K]_{H_{\infty,\theta}}':=\sup_{0<R<1}\sup_{|y|<R} \smallint\limits_{|x|\geq 2R^{1-\theta}}|K(y^{-1}x)-K(x)|dx  <\infty,
    \end{equation} then $T$ admits a bounded extension from $H^1(G)$ into $L^1(G).$

Naturally, the family of compact Lie groups includes  torus of arbitrary dimension. In view of the simplicity of the representation theory on the torus  we separately analyse it in the following subsection. 

\subsection{Oscillating integrals on the torus $\mathbb{T}^n$}
 Let  $G=\mathbb{T}^n\equiv [0,1)^n,$ $0\sim 1,$ be the $n$-dimensional torus $\mathbb{R}^n/\mathbb{Z}^n.$ In this case we have the identification $\widehat{\mathbb{T}}^n=\{e_{\ell}\}_{\ell\in \mathbb{Z}^n}\sim \mathbb{Z}^{n}$ for the unitary dual of the torus. We have denoted $e_{\ell}$ to the exponential function $e_{\ell}(x)=e^{i2\pi \ell\cdot x},$ $x=(x_1,\cdots ,x_n)\in \mathbb{T}^n.$

The spectrum of the Laplace-Beltrami operator $\mathcal{L}_{\mathbb{T}^n}=-\sum_{i=1}^n\partial_{x_i}^2$ is given by
\begin{equation}
 \textnormal{Spect}(\mathcal{L}_{\mathbb{T}^n})=\{4\pi^{2}\ell^2,\,\ell \in \mathbb{Z}^n\}.   
\end{equation}So, all the representations of the torus are one-dimensional and \eqref{LGriem} takes the form
\begin{equation}
    \widehat{\mathcal{L}}_{\mathbb{T}^n}(\ell)=4\pi^{2}\ell^2,\,\,\ell \in \mathbb{Z}^n.
\end{equation}In terms of the Fourier transform of a distribution $K$ on the torus
\begin{equation}
    \widehat{K}(\ell):=\smallint_{\mathbb{T}^n}e_{-\ell}(x)K(x)dx,\,\,\ell \in \mathbb{Z}^n,
\end{equation}the Fourier transform condition \eqref{LxiG} becomes equivalent to the estimate
\begin{equation}\label{LxiT:n}
  \forall \ell\in \mathbb{Z}^n,\,\,  | \widehat{K}(\ell)|\leq C\langle \ell\rangle^{-\frac{n\theta}{2}},\,\ \langle \ell\rangle:=(1+4\pi^2|\ell|^2)^{\frac{1}{2}}\sim |\ell|:=\sqrt{\ell_1^2+\cdots+ \ell_n^2}.
\end{equation}The commutativity of the torus makes \eqref{GS:CZ:cond:2':2} and \eqref{GS:CZ:cond:22':2} equivalent to the following kernel condition 
\begin{equation}\label{Con:kern:torus}
    [K]_{H_{\infty,\theta}}=[K]_{H_{\infty,\theta}}'=\sup_{0<R<1}\sup_{|y|<R} \smallint\limits_{|x|\geq 2R^{1-\theta}}|K(x-y)-K(x)|dx dy  <\infty.
\end{equation}In view of Theorem \ref{main:th:2}, a convolution operator $T$ associated to a convolution kernel $K,$ satisfying the Fourier transform condition \eqref{LxiT:n} and the smoothness condition  \eqref{Con:kern:torus}   admits a bounded extension from  $L^\infty(\mathbb{T}^n) $ into $ BMO(\mathbb{T}^n),$ and   from $H^1(\mathbb{T}^n)$ into $L^1(\mathbb{T}^n).$

\subsection{Oscillating integrals on $\textnormal{SU}(2)\cong \mathbb{S}^3$}\label{SU2}Let us consider the compact Lie group of complex unitary $2\times 2$-matrices  
$$ \textnormal{SU}(2)=\{X=[X_{ij}]_{i,j=1}^{2}\in \mathbb{C}^{2\times 2}:X^{*}=X^{-1}\},\,X^*:=\overline{X}^{t}=[\overline{X_{ji}}]_{i,j=1}^{2}. $$
Let us consider the left-invariant first-order  differential operators $$\partial_{+},\partial_{-},\partial_{0}: C^{\infty}(\textnormal{SU}(2))\rightarrow C^{\infty}(\textnormal{SU}(2)),$$ called creation, annihilation, and neutral operators respectively, (see Definition 11.5.10 of \cite{Ruz}) and let us define 
\begin{equation*}
    X_{1}=-\frac{i}{2}(\partial_{-}+\partial_{+}),\, X_{2}=\frac{1}{2}(\partial_{-}-\partial_{+}),\, X_{3}=-i\partial_0,
\end{equation*}where $X_{3}=[X_1,X_2],$ based on the commutation relations $[\partial_{0},\partial_{+}]=\partial_{+},$ $[\partial_{-},\partial_{0}]=\partial_{-},$ and $[\partial_{+},\partial_{-}]=2\partial_{0},$ the system $X=\{X_1,X_2\}$ satisfies the H\"ormander condition at step $\kappa=2,$ and the Hausdorff dimension defined by the control distance associated to the sub-Laplacian   $\mathcal{L}_{sub}=-X_1^2-X_2^2$  is $Q=4.$ By the compactness of $\textnormal{SU}(2),$ we have that $d_0=Q=4$ and $d_\infty=0.$

We record that the unitary dual of $\textnormal{SU}(2)$ (see \cite{Ruz}) can be identified as
\begin{equation}
\widehat{\textnormal{SU}}(2)\equiv \{ [t_{l}]:2l\in \mathbb{N}, d_{l}:=\dim t_{l}=(2l+1)\}\sim \frac{1}{2}\mathbb{N}.
\end{equation}
There are explicit formulae for $t_{l}$ as
functions of Euler angles in terms of the so-called Legendre-Jacobi polynomials, see \cite{Ruz}.

It was proved e.g.  in \cite{RWT}, that the spectrum of the sub-Laplacian $\mathcal{L}_{sub}$ can be indexed by the sequence
$$ \ell(\ell+1)-m^2,\,\,-\ell\leq m\leq \ell,\,m,\ell\in \frac{1}{2}\mathbb{N}, $$
and then that $\widehat{\mathcal{L}}_{sub}(\ell)$ is given by  the $(2\ell+1)\times (2\ell+1) $-diagonal matrix
\begin{equation}
    \widehat{\mathcal{L}}_{sub}(\ell)=\textnormal{diag}[(\ell(\ell+1)-m^2)]_{-\ell\leq m\leq \ell,m\in \frac{1}{2}\mathbb{N}\,},\,\,\ell\in \frac{1}{2}\mathbb{N}.
\end{equation}Since $d_0=Q=4,$ we have that $\frac{d_0 \theta}{4}=\theta,$ and the condition in \eqref{Fouriergrowth:2} takes the form
 \begin{equation}\label{Fouriergrowth:2SU2}
        \sup_{\ell\in\frac{1}{2} \mathbb{N}}\Vert \textnormal{diag}[(1+\ell(\ell+1)-m^2)^{{{{\theta}} }}]_{-\ell\leq m\leq \ell,m\in \frac{1}{2}\mathbb{N}\,}\times\widehat{K}(\ell) \Vert_{\textnormal{op}}<\infty,
    \end{equation}
    where $\widehat{K}(\ell)=\smallint_{\textnormal{SU}(2)}K(Z)t_{\ell}(Z)^{-1}dZ,$ is the Fourier transform of the group of $K.$
    In view of Theorem \ref{main:th:2}, if  
    $K$ satisfies \eqref{Fouriergrowth:2SU2} and the kernel condition 
    \begin{equation}\label{GS:CZ:cond:2'SU2}
        [K]_{H_{\infty,\theta}}:=\sup_{0<R<1}\sup_{|Y|<R} \smallint\limits_{|X|\geq 2R^{1-\theta}}|K(xy^{-1})-K(x)|dx   <\infty,
    \end{equation}  then $T:L^\infty(\textnormal{SU}(2))\rightarrow BMO(\textnormal{SU}(2))$ admits a bounded extension. On the other hand, if $T$ satisfies the estimate
  \begin{equation}\label{Fouriergrowth:2SU2TT}
        \sup_{\ell\in\frac{1}{2} \mathbb{N}}\Vert \widehat{K}(\ell)\times  \textnormal{diag}[(1+\ell(\ell+1)-m^2)^{{{{\theta}} }}]_{-\ell\leq m\leq \ell,m\in \frac{1}{2}\mathbb{N}\,} \Vert_{\textnormal{op}}<\infty,
    \end{equation}  
and the kernel condition   
    \begin{equation}\label{GS:CZ:cond:22'SU2F}
        [K]_{H_{\infty,\theta}}':=\sup_{0<R<1}\sup_{|Y|<R} \smallint\limits_{|X|\geq 2R^{1-\theta}}|K(y^{-1}x)-K(x)|dx  <\infty,
    \end{equation} then $T:H^1(\textnormal{SU}(2))\rightarrow L^1(\textnormal{SU}(2))$ extends to a bounded operator. We observe that in the kernel conditions \eqref{GS:CZ:cond:2'SU2}  and \eqref{GS:CZ:cond:22'SU2F} the sets $\{Y: |Y|<R\}$ and $\{X:|X|\geq 2R^{1-\theta}\},$ and also the space $(H^1(\textnormal{SU}(2)))'=BMO(\textnormal{SU}(2))$ are defined in terms of the Carnot-Carath\'eodory distance associated to the system of vector fields $\{X_1,X_2\}.$ 
    
\begin{remark}
On the other hand, the Carnot-Carath\'eodory distance associated to the vector fields $\{X_1,X_2,X_3\}$ agrees with the geodesic distance, and (minus) the sums of squares
\begin{equation}
    \mathcal{L}_{\textnormal{SU}(2)}=-X_1^2-X_2^2-X_3^2=-\partial_0^2-\frac{1}{2}[\partial_+\partial_{-}+\partial_{-}\partial_{+}],
\end{equation} is the positive Laplacian on $\textnormal{SU}(2).$
Again, by following e.g.  \cite{Ruz},  the spectrum of the positive Laplacian $\mathcal{L}_{\textnormal{SU}(2)}$ can be indexed by the sequence
$$\lambda_\ell:= \ell(\ell+1),\quad \ell\in \frac{1}{2}\mathbb{N}.$$
Because $n=\dim(\textnormal{SU}(2))=3,$ the Fourier transform condition \eqref{LxiG}  takes the form
 \begin{equation}\label{Fouriergrowth:2SU2'}
      \exists C>0,\,\forall  \ell\in\frac{1}{2} \mathbb{N},\quad \Vert \widehat{K}(\ell) \Vert_{\textnormal{op}}\leq C(1+\ell(\ell+1))^{-\frac{3\theta}{4}}\sim (1+\ell)^{-\frac{3\theta}{2}}.
    \end{equation}
In view of  Theorem \ref{main:th:2}, if  
    $K$ satisfies \eqref{Fouriergrowth:2SU2'} and the kernel condition 
    \begin{equation}\label{GS:CZ:cond:2'SU2TTTTTT}
        [K]_{H_{\infty,\theta}}:=\sup_{0<R<1}\sup_{|Y|<R} \smallint\limits_{|X|\geq 2R^{1-\theta}}|K(xy^{-1})-K(x)|dx   <\infty,
    \end{equation}  then $T:L^\infty(\textnormal{SU}(2))\rightarrow BMO(\textnormal{SU}(2))$ admits a bounded extension. On the other hand, if $T$ satisfies \eqref{Fouriergrowth:2SU2'} 
and the kernel condition   
    \begin{equation}\label{GS:CZ:cond:22'SU2FTT}
        [K]_{H_{\infty,\theta}}':=\sup_{0<R<1}\sup_{|Y|<R} \smallint\limits_{|X|\geq 2R^{1-\theta}}|K(y^{-1}x)-K(x)|dx  <\infty,
    \end{equation} then $T:H^1(\textnormal{SU}(2))\rightarrow L^1(\textnormal{SU}(2))$ extends to a bounded operator. Note that the conditions  \eqref{GS:CZ:cond:2'SU2TTTTTT}
    and \eqref{GS:CZ:cond:22'SU2FTT}, the sets $\{Y: |Y|<R\}$ and $\{X:|X|\geq 2R^{1-\theta}\},$ and also the space $(H^1(\textnormal{SU}(2)))'=BMO(\textnormal{SU}(2))$ are defined in terms of geodesic distance.
\end{remark}
    
Now, we are going to present some applications of Theorem  \ref{main:th:2} to (the non-compact case of) graded Lie groups. We start with the Euclidean case.

\subsection{Oscillating integrals on $ \mathbb{R}^n$}
 The unitary dual of $\mathbb{R}^n$ admits the identification $\widehat{\mathbb{R}}^n=\{e_{\xi}:\xi\in \mathbb{R}^n\}\sim \mathbb{R}^n, $ where $e_{\xi}(x):=e^{i2\pi x\cdot \xi}.$ The Fourier transform of a distribution $K$ is given by $$\widehat{K}(\xi)=\smallint_{\mathbb{R}^n}e_{-\xi}(x)K(x)dx,$$ for any $\xi\in \mathbb{R}^n.$ The model Rockland operator $\mathcal{R}$ on $\mathbb{R}^n$ is the positive Laplacian
$ \mathcal{R}=-\Delta_x:=-\partial_{x_1}^2-\cdots - \partial_{x_n}^2,
$  which is homogeneous of degree $\nu=2.$ In this case \eqref{pi:R} becomes to be
\begin{equation*}
    e_{\xi}(-\Delta_x)=4\pi^2|\xi|^2,\,\quad \xi \in \mathbb{R}^n.
\end{equation*}The Fourier transform condition in \eqref{Fourier:growth:1} takes the form
\begin{equation}\label{Ft;C;Rn}
        \sup_{\xi\in \mathbb{R}^n}\vert e_\xi\left((1-\Delta_x)^{\frac{n\theta}{4}}\right)\widehat{K}(\xi) \vert=  \sup_{\xi\in \mathbb{R}^n}|(1+4\pi^2|\xi|^2)^{\frac{n\theta}{4}}\widehat{K}(\xi)|<\infty.
    \end{equation}Note that the  condition in \eqref{Ft;C;Rn} is equivalent to the existence of $C>0$ satisfying that
    \begin{equation}\label{Fefferman:cëx}
   \forall \xi\in \mathbb{R}^n,\,     |\widehat{K}(\xi)|\leq C |(1+4\pi^2|\xi|^2)^{-\frac{n\theta}{4}}\sim C (1+|\xi|)^{-\frac{n\theta}{2}}.
    \end{equation} The commutativity of $\mathbb{R}^n$ makes \eqref{GS:CZ:cond:2} and \eqref{GS:CZ:cond:22} equivalent to the following kernel condition 
\begin{equation}\label{Con:kern:rn}
    [K]_{H_{\infty,\theta}}=[K]_{H_{\infty,\theta}}'=\sup_{0<R<1}\sup_{|y|<R} \smallint\limits_{|x|\geq 2R^{1-\theta}}|K(x-y)-K(x)|dx   <\infty.
\end{equation}In view of Theorem \ref{main:th:2}, a convolution operator $T$ associated to a convolution kernel $K,$ satisfying the Fourier transform condition \eqref{Fefferman:cëx} and the smoothness condition  \eqref{Con:kern:rn}   admits a bounded extension from  $L^\infty(\mathbb{R}^n) $ into $ BMO(\mathbb{R}^n),$ and   from $H^1(\mathbb{R}^n)$ into $L^1(\mathbb{R}^n).$

\subsection{Oscillating integrals on the Heisenberg group $\mathbb{H}_n$}
 Let us consider the Heisenberg group $\mathbb{H}_n$ as the manifold $\mathbb{R}^{2n+1}$ endowed of the product
\begin{equation}
    (x,y,t)\cdot (x',y',z')=(x+x',y+y',t+t'+\frac{1}{2}(xy'-x'y)),\,(x,y,z),(x',y',z')\in \mathbb{H}_n.
\end{equation} The Lie algebra $\mathfrak{h}_n=\textnormal{Lie}(\mathbb{H}_n)$ is spanned by the vector-fields
\begin{equation}
    X_j=\partial_{x_j}-\frac{y_j}{2}\partial_{t},\,\quad Y_j=\partial_{x_j}+\frac{x_j}{2}\partial_{t},\,1\leq j\leq n,\,\quad T=\partial_t. 
\end{equation}The canonical relations are $[X_j,Y_j]=T.$ The positive sub-Laplacian $\mathcal{L}_{sub}$ on the Heisenberg group
$$ \mathcal{L}_{sub}:=-\sum_{j=1}^nX_j^2+Y_j^2=-\sum_{j=1}^n\left(\partial_{x_j}-\frac{y_j}{2}\partial_{t}\right)^2+\left(\partial_{x_j}+\frac{x_j}{2}\partial_{t}\right)^2, $$ is a Rockland operator of homogeneous degree $\nu=2$ and the homogeneous dimension of $\mathbb{H}_n$ is  $Q=2n+2.$ We observe that if $\mathbb{H}_n$ is endowed with the Carnot-Carath\'eodory distance $|\cdot|$ then the local dimension $d_0$ agrees with the global dimension $d_\infty$ and $d_0=d_\infty=Q=2n+2.$ Moreover, on any nilpotent stratified group ones always has that $d_0=d_\infty.$ 

The unitary dual $\widehat{\mathbb{H}}_n$ of $\mathbb{H}_n$ admits the identification
\begin{equation}\label{dual:of:Hn}
   \widehat{\mathbb{H}}_n=\{\pi_{\lambda}:\lambda\neq 0,\,\lambda\in \mathbb{R}\}\sim \mathbb{R}^{*}:=\mathbb{R}\setminus \{0\}. 
\end{equation} In \eqref{dual:of:Hn}, the family $\pi_{\lambda},$ $\lambda\in \mathbb{R}^*,$ are the Schr\"odinger representations, which are the unitary operators on $L^2(\mathbb{R}^n)$ defined via 
\begin{equation*}
    \pi_\lambda (x,y,z) h(u)=e^{i\lambda (t+\frac{1}{2}x\cdot y)}e^{i\sqrt{\lambda}yu}h(u+\sqrt{|\lambda|}x),\,\sqrt{\lambda}:=\frac{\lambda}{|\lambda|}\sqrt{|\lambda|},\,\lambda\in \mathbb{R}^*.
\end{equation*}
In the Heisenberg group setting, for $\mathcal{R}=\mathcal{L},$ \eqref{pi:R} takes the form of the dilated harmonic oscillator
\begin{equation*}
   \pi_\lambda(\mathcal{L}_{sub})=|\lambda|(-\Delta_u+|u|^2), \, \Delta_u:=\sum_{j=1}^n\partial_{u_j}^2,\,u=(u_1,\cdots,u_n)\in \mathbb{R}^n.
\end{equation*}
For a distribution $K$ on $G=\mathbb{H}_n,$ the condition in \eqref{Fourier:growth:1} for the Fourier transform of $K,$
$$ \widehat{K}(\lambda)=\smallint_{\mathbb{H}_n}K(x,y,t)\pi_\lambda(x,y,t)^{-1}d(x,y,t):L^2(\mathbb{R}^n)\rightarrow L^2(\mathbb{R}^n), $$
becomes equivalent to the following Fourier transform condition
\begin{equation}\label{FT:hn:cond}
   \sup_{\lambda\in \mathbb{R}^*}\Vert \pi_\lambda(1+\mathcal{L})^{\frac{Q\theta}{2\nu}}\widehat{K}(\lambda) \Vert_{\textnormal{op}}=\sup_{\lambda\in \mathbb{R}^*}\Vert \left(1+|\lambda|(-\Delta_u+|u|^2)\right)^{\frac{(n+1)\theta}{2}}\widehat{K}(\lambda) \Vert_{\textnormal{op}}<\infty.
\end{equation}
In view of Theorem \ref{main:th:2}, a convolution operator $T$ (defined via $Tf=f\ast K,$ $f\in C^\infty(\mathbb{H}_n)$) associated to a right-convolution kernel $K,$ satisfying  \eqref{FT:hn:cond} and the kernel condition
\begin{equation}
      \sup_{0<R<1}\sup_{\{Y\in \mathbb{H}_n:|Y|<R\}} \smallint\limits_{\{X\in \mathbb{H}_n: |X|\geq 2R^{1-\theta} \}}|K(xy^{-1})-K(x)|dx   <\infty,
    \end{equation}  admits a bounded extension from  $L^\infty(\mathbb{H}_n) $ into $ BMO(\mathbb{H}_n).$ Moreover, if $K$ satisfies  that
    \begin{equation}\label{FT:hn:cond222}
   \sup_{\lambda\in \mathbb{R}^*}\Vert \widehat{K}(\lambda) \pi_\lambda(1+\mathcal{L})^{\frac{Q\theta}{2\nu}} \Vert_{\textnormal{op}}=\sup_{\lambda\in \mathbb{R}^*}\Vert \widehat{K}(\lambda) \left(1+|\lambda|(-\Delta_u+|u|^2)\right)^{\frac{(n+1)\theta}{2}}\Vert_{\textnormal{op}}<\infty,
\end{equation}
     and the kernel condition
    \begin{equation}
        \sup_{0<R<1}\sup_{\{Y\in \mathbb{H}_n:|Y|<R\}} \smallint\limits_{\{X\in \mathbb{H}_n: |X|\geq 2R^{1-\theta} \}}|K(y^{-1}x)-K(x)|dx  <\infty,
    \end{equation} then $T$ admits a bounded extension from the Hardy space $H^1(\mathbb{H}_n)$ into $L^1(\mathbb{H}_n).$

\subsection{Oscillating operators on Lie groups of polynomial growth and of type I}\label{LGPGTI}
Note that Theorem   \ref{main:th:2} has been proved for any Lie group of polynomial growth. However, if we add the hypothesis that $G$ is also of type I, an explicit condition can be given in terms of the Fourier analysis associated to the sub-Laplacian in order to express the $L^2$-boundedness hypothesis in \eqref{Fouriergrowth:2}. However, to do this we require some aspects of the functional calculus for the sub-Laplacian $\mathcal{L}.$ We record it in the following remark.

\begin{remark}[Functional calculus for the sub-Laplacian revisited] Let $\mathcal{L}=-\sum_{i=1}^kX_i^2$ be a positive sub-Laplacian associated to a H\"ormander system of vector-fields $\{X_i\}_{i=1}^k.$ Then $\mathcal{L}$ is formally self-adjoint as an element of the universal enveloping algebra $\mathfrak{u}(\mathfrak{g}).$ Moreover, for a.e. $\xi,$ $\widehat{\mathcal{L}}(\xi)=\sum_{i=1}^kd\xi(X_k)^2$ is essentially self-adjoint on $H_\xi$ and we keep the same notation for its self-adjoint extension. Let $E$ and $E_\xi$ be the spectral measures of $\mathcal{L}$ and $\widehat{\mathcal{L}}(\xi),$ respectively.

Let $B$ be a Borel subset $B\subset \mathbb{R}$ and let $E(B):L^2(G)\rightarrow
L^2(G)$ and $E_\xi(B):H_\xi\rightarrow H_\xi$ be the corresponding projections. In particular $E(B)$ is a left-invariant operator. The group Fourier transform of its convolution kernel $E(B)\delta$ is 
\begin{equation}\label{B}
    \mathscr{F}[E(B)\delta](\xi)=E_\xi(B),\,[\xi]\in \widehat{G}.
\end{equation}One can prove \eqref{B} by following the argument in \cite[Page 182]{FischerRuzhanskyBook}.
Indeed, if $F_\xi:B\mapsto \mathscr{F}[E(B)\delta](\xi), $ then it is not difficult to prove that $F_\xi$ is a spectral resolution on $H_\xi.$
Note that
if $\phi$ is a measurable function on $G,$ then the spectral multiplier $\phi(\mathcal{L})$ is defined via 
$$\phi(\mathcal{L}):=\smallint\limits_{0}^\infty \phi(\lambda) dE(\lambda).$$ Note that $\textnormal{Dom}(\phi(\mathcal{L}))=\{f\in L^2(G):\smallint\limits_{0}^\infty|\phi(\lambda)|^2d\Vert E(\lambda) f\Vert_{L^2(G)}^2<\infty\}.$ On the other hand, if $\phi(\mathcal{L})\delta$ is the right-convolution kernel of $\phi(\mathcal{L}),$ then for any $f\in C^\infty_0(G),$
\begin{equation}\label{F(B)}
    \mathscr{F}[\phi(\mathcal{L})f](\xi) =   \mathscr{F}[\phi(\mathcal{L})\delta](\xi)\widehat{f}(\xi)=\mathscr{F}[\smallint\limits_0^\infty\phi(\lambda)dE(\lambda)f](\xi)=\smallint\limits_0^\infty\phi(\lambda)dF_\xi(\lambda)\widehat{f}(\xi),
\end{equation}
 with $\phi=1_B$  the characteristic function of $B.$ Note that \eqref{F(B)} also holds for a finite linear combination of characteristic functions,
and then, passing through the limit carefully, for any $\phi\in L^\infty(G),$ with $f\in L^2(G)$ or $\phi(\lambda)=\lambda$ for $f\in C^\infty_0(G).$

Consequently,
\begin{align*}
  \mathscr{F}[\mathcal{L}f](\xi) =  \smallint\limits_{0}^\infty\lambda dF_\xi(\lambda)\widehat{f}(\xi)=\mathscr{F}[\smallint\limits_{0}^\infty\lambda dE(\lambda) f](\xi)=\widehat{\mathcal{L}}(\xi)\widehat{f}(\xi),\,f\in C^\infty_0(G),
\end{align*}
{{where $\widehat{\mathcal{L}}(\xi):=\mathscr{F}[\mathcal{L}\delta](\xi),$ $[\xi]\in \widehat{G}.$}}
In view of the Dixmier-Malliavin theorem (see Theorem 1.7.8 of \cite{FischerRuzhanskyBook}), the space of smooth vectors $H_\xi^\infty$ is spanned by elements of the form $\widehat{f}(\xi)v,$ $v\in H_\xi^\infty,$ (in the sense that any vector in $H_\xi^\infty$ is a finite linear combination of vectors of the form $\widehat{f}(\xi)v$), and we have that
\begin{equation*}
    \widehat{\mathcal{L}}(\xi)v=\smallint\limits_{0}^\infty\lambda dF_\xi(\lambda)v,\,v\in H_\xi^\infty.
\end{equation*}
In view of the density of $H_\xi^\infty$ in $H_\xi$ (see Theorem 1.7.7 of \cite{FischerRuzhanskyBook}) and the uniqueness of the spectral measure, we have that  $F_\xi=E_\xi$ proving \eqref{B}.

Note that one of the fundamental consequences of \eqref{B}, is that for a measurable function $\phi$ on $\mathbb{R},$ one has the equality of operators
\begin{equation}\label{FBdelta}
\mathscr{F}[\phi(\mathcal{L})\delta](\xi)= \mathscr{F}(\smallint\limits_{0}^\infty \phi(\lambda) dE(\lambda)\delta)(\xi)=\smallint\limits_{0}^\infty \phi(\lambda) dE_\xi(\lambda)=\phi(\widehat{\mathcal{L}}(\xi)),    
\end{equation} on the class of smooth vectors $v\in H_\xi^\infty,$ such that $\smallint\limits_{0}^\infty|\phi(\lambda)|^2d\Vert E_\xi(\lambda) v\Vert_{H_\xi}^2<\infty.$
In particular, for any $s\in \mathbb{R},$ in the same class of smooth vectors one has the equality $\mathscr{F}[(1+\mathcal{L})^{\frac{s}{2}}\delta]=(1+\widehat{\mathcal{L}}(\xi))^{\frac{s}{2}}.$

\end{remark}
\begin{remark}
Now we discuss the hypothesis \eqref{Fouriergrowth:2} when the group $G$  is of polynomial growth and of type I. In view of \eqref{FBdelta}, via the Fourier transform we have
\begin{equation*}
   (1+\mathcal{L})^{\frac{d_0\theta}{4}}Tf(x)=\smallint_{\widehat{G}}\textnormal{Tr}[\xi(x)(1+\widehat{\mathcal{L}}(\xi))^{\frac{d_0\theta}{4}}\widehat{K}(\xi)\widehat{f}(\xi)]d\nu(\xi),\,f\in C^{\infty}_0(G). 
\end{equation*}  So, in view of the Plancherel formula, the estimate
\begin{equation}\label{FT:Cond:growth}
    \sup_{[\xi]\in \widehat{G}}\Vert (1+\widehat{\mathcal{L}}(\xi))^{\frac{d_0\theta}{4}}\widehat{K}(\xi) \Vert_{\textnormal{op}}<\infty,
\end{equation}is equivalent to the condition \eqref{Fouriergrowth:2}. In view of Theorem  \ref{main:th:2}, the Fourier transform condition \eqref{FT:Cond:growth}  on  $K$  together with oscillating condition
 \begin{equation}
        [K]_{H_{\infty,\theta}}:=\sup_{0<R<1}\sup_{|y|<R} \smallint\limits_{|x|\geq 2R^{1-\theta}}|K(xy^{-1})-K(x)|dx   <\infty,
    \end{equation}
imply the boundedness of $T$ from  $L^\infty(G)$ into $BMO(G).$ Moreover, if $K$ satisfies the following estimate
\begin{equation}\label{FT:Cond:growth2}
    \sup_{[\xi]\in \widehat{G}}\Vert \widehat{K}(\xi)(1+\widehat{\mathcal{L}}(\xi))^{\frac{d_0\theta}{4}} \Vert_{\textnormal{op}}<\infty,
\end{equation}
and the kernel condition  
\begin{equation}
        [K]_{H_{\infty,\theta}}':=\sup_{0<R<1}\sup_{|y|<R} \smallint\limits_{|x|\geq 2R^{1-\theta}}|K(y^{-1}x)-K(x)|dx  <\infty,
    \end{equation} then $T:H^1(G)\rightarrow L^1(G)$ extends to a bounded operator.
\end{remark}
\noindent{\textbf{Conflict of interests statement.}}   On behalf of all authors, the corresponding author states that there is no conflict of interest.

\bibliographystyle{amsplain}

\end{document}